\documentclass[a4paper,12pt]{article}

\usepackage{amsmath,amsthm,amssymb,sansmath,color,geometry,multicol,graphicx,float,color,authblk, hyperref,stmaryrd,calc,enumitem}
\usepackage[english]{babel}
\usepackage[utf8]{inputenc}
\newtheorem{thm}{Theorem}
\newtheorem{prop}[thm]{Proposition}

\newtheorem{lem}[thm]{Lemma}

\newtheorem{hyp}{Assumption}

\newtheorem{defi}{Definition}

\newcommand{\po}{\left(}
\newcommand{\pf}{\right)}
\newcommand{\co}{\left[}
\newcommand{\cf}{\right]}
\newcommand{\cco}{\llbracket}
\newcommand{\ccf}{\rrbracket}
\newcommand{\R}{\mathbb R}
\newcommand{\Z}{\mathbb Z}
\newcommand{\T}{\mathbb T}

\newcommand{\Su}{\mathbb S^{1}}
\newcommand{\E}{\mathbb E}
\newcommand{\EE}{\mathcal E}

\newcommand{\dd}{\text{d}}

\newcommand{\na}{\nabla}
\newcommand{\TM}{T\mathcal M}

\title{Weakly self-interacting velocity jump processes for bacterial chemotaxis and adaptive algorithms.}
\author{Pierre Monmarché\footnote{Sorbonne Universit\'es, UPMC Paris 06, Laboratoire Jacques-Louis Lions UMR CNRS 7598,
 F75005 Paris, France}}

\begin{document}

\maketitle

\abstract{Self-interacting velocity jump process are introduced, which behave in  large time similarly to the corresponding self-interacting diffusions, namely the evolution of their normalized occupation measure approaches a deterministic flow.
\textbf{MSC-class:} 60F99, 60J75}

\section{Introduction}

Rather than by a diffusion process, the motion of a bacterium in a gradient of chemo-attractors may be modelled  (see \cite{ErbanOthmer,Calvez,FontbonaGuerinMalrieu2016} and references within) by a velocity jump process: the particle runs straight ahead at constant speed for some time, until it decides, depending on its environment, to change direction, which is done in a tumble phase which is short enough with respect to the run one to be considered instantaneous.

In the present work we add to this model  a self-interacting mechanism, namely we suppose the process  is influenced by its past trajectory. Among the many ways to add self-interaction and memory to an initially Markovian dynamic (see the survey \cite{Pemantle2007} for instance), we will consider a weak self-interaction such as introduced in \cite{BenaimLedouxRaimond} for the  diffusion
\begin{eqnarray}\label{EqDefDiffuInter}
\dd X_t &=& -\po\frac1t\int_0^t \nabla V(X_t - X_s) \dd s\pf \dd t + \sqrt 2 \dd B_t,
\end{eqnarray}
or more generally  a self-interaction that depends on the normalized occupation measure
\begin{eqnarray*}
\mu_t &=& \frac1t\int_0^t \delta_{X_s}\dd s.
\end{eqnarray*}
Note that a strong self-interaction, for which by contrast the drift is a function of the non-normalized occupation measure $t\mu_t$, such as studied in \cite{Raimond1997,Benaim2014} for diffusions, is studied in the case of a velocity jump process in \cite{MonmarcheGauthier}.

 We are interested in the long-time behaviour of the process, and in particular in the question of the influence of the weak self-interaction on this long-time behaviour: if the process tends to go back to where it has already been, is the interaction sufficient to confine it in some localized place ? In particular, if the initial landscape is symmetric, is the interaction strong enough to break the symmetry ? Beyond the modelling question, self-interaction is also used in stochastic algorithms  (see e.g. \cite{BenaimBrehier2016} and Section \ref{SectionAdaptive} for the ABP algorithm). In practice, for such algorithms, the underlying Markov process is often a kinetic one rather than an overdamped Langevin diffusion
\begin{eqnarray}\label{EqDifussionClassique}
\dd X_t &=& -\nabla V(X_t) \dd t + \sqrt 2 \dd B_t,
\end{eqnarray}
which is nevertheless used in the theoretical proofs of convergence for the algorithms. In particular, the use of velocity jump processes in stochastic algorithms have recently gained much interest (\cite{BierkensRoberts,PetersdeWith,MonmarcheRTP}). To our knowledge, the present work is the first time a convergence result is established for a weakly self-interacting kinetic process (since the release of the first version of the present work, Benaim and Brehier \cite{BenaimBrehier2017} have also studied the case of the Langevin process, a kinetic diffusion).

\bigskip

First, we recall the definition of the Markovian velocity jump process, which is in some sense a non-diffusive analoguous of the diffusion \eqref{EqDifussionClassique}.

\subsection{The Markovian velocity jump process}

Let $\mathcal M$ be a compact connected smooth Riemanian manifold of dimension $d$ with no boundary, and $\TM$ be its tangent bundle. For $r\in [0,\infty)$ and $R\in (0,\infty]$ with $r\leqslant R$, let $\EE = \{(x,y)\in\TM,\ r\leqslant |y| \leqslant R\}$. 
Let $(t,x,y)\in \R_+\times \EE\mapsto \varphi_t(x,y)\in \EE$ be the (restriction on $\EE$ of the) flow associated to the exponential map on $\TM$, defined as follows: for $x\in \mathcal M$ and $y\in T_x\mathcal M$, there exists a unique geodesic $\gamma$ on $\mathcal M$ with $\gamma(0) = x$ and $\gamma'(0) = y$. Then we set
\begin{eqnarray*}
\varphi_t(x,y) \ = \ \po \varphi_t^{(1)}(x,y),\varphi_t^{(2)}(x,y) \pf & :=& \po \gamma(t),\gamma'(t)\pf.
\end{eqnarray*}
Since $\mathcal M$ is compact, it is geodesically complete, meaning that $\varphi_t$ is defined for all $t\geqslant 0$.

For example, on the $d$-dimensional torus $\mathbb T^d$, this simply reads
\begin{eqnarray*}
\varphi_t(x,y) & =& \po x+ty,y\pf.
\end{eqnarray*}
On the $d$-dimensional sphere $\mathbb S_d$,
\begin{eqnarray*}
\varphi_t(x,y) & =& \po x \cos(|y| t) + \frac{y}{|y|}\sin\po |y| t\pf, -|y|x\sin(|y|t) + y \cos(|y|t)\pf.
\end{eqnarray*}

A velocity jump process $Z=(X,Y)\in \EE$ is a piecewise deterministic Markov process (PDMP; see \cite{MalrieuPDMP} for general considerations on PDMP) that follows the flow $\varphi$ up to random times where the velocity $Y$ jumps to a new value. The jump mechanism is defined from a jump rate $\lambda : \EE\rightarrow \R_+$, and a jump kernel $H: \EE\rightarrow \mathcal P \po \EE\pf$ (where $\mathcal P(F)$ denotes the set of probability measure of $F$). We suppose that $\lambda$ is continuous and bounded, and that $H$ is such that, if $(U,V)$ is a random variable with law $H(x,y)$, then $U=x$ almost surely (in other words, only the velocity jumps). We still denote $H$ the Markov operator such that $Hf(x,y) = \mathbb E\po f(U,V)\ |\ (U,V)\sim H(x,y)\pf$, and we write
\begin{eqnarray*}
H f(x,y) & = & \int f(x,v) h(x,y,\dd v).
\end{eqnarray*}

\bigskip

\noindent \textbf{Construction of the process.} Suppose that the process has been defined up to a time $t_0$. Set
  \begin{eqnarray*}
 t_1 & = & \inf\left\{ t>t_0,\ E<\int_{t_0}^t \lambda\po\varphi_s(Z_{t_0}\pf\dd s\right\}
 \end{eqnarray*}
to be the next jump time, where $E$ is a random variable with standard (i.e. mean 1) exponential law, independent from the past of the process. Then, set $Z_t=(X_t,Y_t) = \varphi_{t-t_0}\po Z_{t_0}\pf$ for $t\in[t_0,t_1)$, and  draw $Z_{t_1}$ according to  $H\po \varphi_{t_1-t_0}\po Z_{t_0}\pf\pf$. 
The process is thus defined up to time $t_1$, hence up to any jump time $t_n$. Since $\lambda$ is bounded, there cannot be an infinite number of jumps in a finite time interval, so that the process is defined for all time.

\bigskip

Let $(P_t)_{t\geq 0}$ be the Markov semi-group associated to $Z$, namely
\begin{eqnarray*}
P_t f(z) &:=& \mathbb E\po f(Z_t)\ |\ Z_0 = z\pf,
\end{eqnarray*}
on  functions $f\in L^\infty(\EE)$. By duality, it acts on $m\in \mathcal P\po\EE\pf$ by $(mP_t)f=m(P_t f)$, where we write $mf=\int f \dd m$. Recall that its infinitesimal generator  is defined by
\begin{eqnarray*}
L f(z) &:=& \po \partial_t \pf_{| t=0}  P_t f(z) 
\end{eqnarray*}
whenever this derivative exists.  Here, for any smooth function $f$ on $\TM$, we have 
\begin{eqnarray}\label{EqGeneMarkov}
L f(x,y)\ &=& D f(z) + \lambda\po z\pf \po   Hf(z)   - f(z)\pf,
\end{eqnarray} 
where 
\[D f(z)\ =\  \underset{t\rightarrow 0} \lim  \frac{f\po\varphi_t(z)\pf - f(z)}{t}.\]
For instance, on the torus, $Df(x,y) = y \cdot \nabla_x f(x,y)$.

\bigskip

It can be seen that the set $\mathcal C^1_b(\TM)$ of smooth and bounded functions on $\TM$ is a core for $L$. Indeed, if we suppose, in the first instance, that $\lambda \in \mathcal C^1_b(\TM)$ and that $H$ fixes $\mathcal C^1_b(\TM)$, then so does $P_t$ for $t\geqslant 0$, as can be seen with \cite[Equation (7)]{BLBMZ3} and dominated convergence arguments. From \cite[Proposition 19.9]{Kallenberg}, $\mathcal C^1_b(\TM)$ is a core for $L$. Now, in the case where $\lambda$ is only continuous, an approximation argument (see \cite[Theorem 19.25]{Kallenberg}) concludes.

\bigskip

\noindent \textbf{Remark.} It is possible to define a velocity jump process on a smooth compact Riemanian manifold with smooth boundary, by requiring that the process is reflected at the boundary (like a deterministic billard). However, in this paper, we are mainly interested in the cases of the torus and of the sphere.


 
\subsection{The self-interacting process}

Now we suppose that for each $\nu \in \mathcal P \po \mathcal M\pf$, $\lambda^\nu$ and $H^\nu$ (and $h^\nu$) are a jump rate and a jump kernel on $\EE$ that satisfy the assumptions of the previous section. We suppose that $\nu \mapsto \lambda^\nu$ and $\nu\mapsto H^\nu$  are continuous with respect to the weak topology on $\mathcal P(\mathcal M)$ and the uniform topology on $\mathcal C^0(\EE)$. We denote by $(P_t^\nu)_{t\geqslant 0}$ and $L^\nu$ the associated semi-groups and generators. When needed and without ambiguity, we will sometimes write $\lambda(z,\nu) = \lambda^\nu(z)$.


Let $Z=(X,Y)$ be a measurable process on $\EE$ (namely a measurable function from some probability space $\Omega$ to the set of c\`adl\`ag functions on $\EE$ endowed with the Skorokhod topology), $r>0$, $m_0\in\mathcal P(\EE)$ and $\mu_0\in \mathcal P(\mathcal M)$. We call 
\begin{eqnarray*}
\mu_t & := & \frac{r \mu_0 + \int_0^t \delta_{X_s} \dd s}{r+t}
\end{eqnarray*}
the (normalized) occupation measure of $X$ at time $t$ with initial weight $r$ and initial value $\mu_0$.  In other words, $\mu_t$ is the probability measure on $\mathcal M$ defined by
\[\int f \dd \mu_t = \frac{r}{r+t} \int f \dd \mu_0 + \frac{1}{r+t}\int_0^t f(X_s)\dd s.\]
 Note that only the position $X$ is concerned, and not the velocity $Y$. We denote by $\po \mathcal F_t\pf_{t\geq 0}$ the filtration associated to $(Z_t)_{t\geq 0}$.
\begin{defi}
We say $Z$ (or equivalently $(Z,\mu)$) is a self-interacting velocity jump process (SIVJP) with parameters $r$, $\mu_0$, $m_0$, 
 if the law of $Z_0$ is  $m_0$ and if for all $f\in\mathcal C^1_b(\TM)$ and all $z\in \EE$,
\begin{eqnarray*}
M_t^f & := & f(Z_t) - f(Z_0) - \int_0^t L^{\mu_s} f (Z_s)\dd s 
\end{eqnarray*}
is an $\mathcal F_t$-martingale.
\end{defi}
All or part of the parameters may be omitted when there is no ambiguity.

\bigskip

\noindent \textbf{Remark.} The martingale bracket of $M_t^f$ is classically derived from the carr\'e du champ operator $\Gamma^{\nu} f:=\frac12 L^{\nu} f^2 - f L^\nu f$ as
\begin{eqnarray*}
[M_t^f,M_t^f] &=&  2\int_0^t   \Gamma^{\mu_s} f(Z_s)  \dd s
\end{eqnarray*}
where here
\begin{eqnarray*}
\Gamma^{\nu} f(x,y) &=&    \frac12 \lambda^\nu\po x,y \pf\int \po f(x,v) - f(x,y)\pf^2 h^{\nu}(x,y,\dd v).
\end{eqnarray*}
We still denote by $\Gamma^{\nu}$ the associated symmetric bilinear form,
\begin{eqnarray*}
\Gamma^\nu(f,g)(x,y) & :=& \frac12 \po L^\nu(fg) - f L^\nu g - gL^\nu f\pf\\
&=& \frac12 \lambda^\nu\po x,y \pf\int \po f(x,v) - f(y)\pf \po g(x,v) - g(x,y)\pf h^{\nu}(x,y,\dd v).
\end{eqnarray*}

\bigskip

 For $r>0$, $\nu \in \mathcal P\po\mathcal M\pf$, $z\in \EE$ and $t\geq 0$ we write 
\begin{eqnarray*}
\Phi_{r,t}\po x,y,\nu \pf &=& \po \varphi_t(x,y) ,\  \frac{r \nu + \int_{0}^t \delta_{\varphi_s^{(1)}(x,y)}\dd s}{r+t}\pf.
\end{eqnarray*}
Note that
\begin{eqnarray*}
 \Phi_{r,t_0+t}(x,y,\nu) &=& \Phi_{r+t_0,t}\po \Phi_{r,t_0}(x,y,\nu)\pf.
\end{eqnarray*}
In other words, the initial weight $r$ can be interpreted as an initial break-in time, only after which the occupation measure is updated.

An SIVJP can be constructed as follows: from a time $t_0$ the process $(Z_t,\mu_t)$ evolves deterministically along the flow $\Phi_{r+t_0,\cdot}$ up to the next jump time $t_1$ which is defined, thanks to a standard exponential r.v. $E$,  as  
\begin{eqnarray*}
t_1 &=& \inf\left\{t>t_0,\ E < \int_{t_0}^t  \lambda \po \Phi_{r+t_0,s}\po Z_{t_0},\mu_{t_0}\pf\pf \dd s\right\}.
\end{eqnarray*}
At time $t_1$, $Z$ is drawned according to the law $H^{\mu_{t_1}}\po\varphi_{t_1-t_0}(Z_{t_0})\pf$. In other words, the whole process $(Z,\mu)$ is an inhomogeneous PDMP, whereas $Z$ alone is not a Markov process.  Given the velocity $(Y_t)_{t\geq 0}$, the position $X$ and the occupation measure $\mu$ are completely deterministic with
\begin{eqnarray*}
X_t &=& x + \int_0^t Y_s \dd s\\
\mu_t &=&  \frac{r \nu + \int_0^t \delta_{X_s} \dd s}{r+t}.
\end{eqnarray*}

\subsection{Main results}

We will work under the following assumptions:
\begin{hyp}\label{HypoUnifnu}
There exist $\lambda_{max} \geqslant \lambda_{min} >0$ such that for all $\nu\in\mathcal P\po\mathcal M\pf$ and all $z\in\EE$,
\[\lambda_{min}\ <\ \lambda^\nu(z) \ < \ \lambda_{max}.\]
There exist $c\in(0,1)$ and a probability $p$ on $(0,\infty)$ such that, 
for all $\nu\in\mathcal P\po\mathcal M\pf$ and for all positive, bounded $f$ on $\TM$ and $(x,y)\in\EE$,
\[ H^\nu f(x,y) \ \geqslant \ c \int  f(x,r\theta) p(\dd r)\dd \theta,\]
where $\dd \theta$ here stands for the uniform law on the unit ball of $T_x\mathcal M$.
\end{hyp}
This means that, whatever $x,y,\nu$, at a constant rate, the velocity is refreshed to a completely new one, isotropic.


\bigskip

In the following, we denote by $d_{TV}$ the total variation distance between probability measures,
\begin{eqnarray*}
d_{TV}\po\nu_1,\nu_2\pf &:=& \inf\left\{ \mathbb P\po V_1\neq V_2\pf,\ Law(V_i) = \nu_i,\ i=1,2\right\}.
\end{eqnarray*}

\begin{hyp}\label{HypoContiNu}
There exists $C>0$ such that 
\begin{eqnarray*}
\left\|\lambda^{\nu_1} - \lambda^{\nu_2}\right\|_\infty & \leqslant & C  d_{TV} \po \nu_1, \nu_2\pf
\end{eqnarray*}
and for all bounded $f$ on $\EE$,
\begin{eqnarray*}
\left\|H^{\nu_1}f - H^{\nu_2} f\right\|_\infty & \leqslant & C  d_{TV} \po \nu_1, \nu_2\pf \| f\|_\infty
\end{eqnarray*}
\end{hyp}
In other words, $\nu \mapsto \lambda^\nu,H^\nu$ are more than continuous: they are Lipschitz maps.

\bigskip

As will be proven in Lemma \ref{LemMesureInvariante} below, Assumption \ref{HypoUnifnu} implies that, $\nu$ being fixed, the Markov process with generator $L^\nu$ admits a unique invariant measure. We write $\Pi(\nu)$ the latter, and $\pi(\nu)$ its marginal on $\mathcal M$, namely
\begin{eqnarray*}
\pi(\nu)(x) & = & \int_{T_x\mathcal M} \Pi(\nu)\po x,\dd v \pf. 
\end{eqnarray*}
It will also be proven in Lemma \ref{LemMesureInvariante} that $\pi(\nu)$ admits a positive  density with respect to the Lebesgue measure, still denoted $\pi(\nu)$. 

\bigskip

Consider a SIVJP $(Z_t,\mu_t)_{t\geq 0}$. If   $\mu_t$ were to converge to some law $\mu_\infty$, then for large times $Z$ should more or less behave as a Markov process with generator $L^{\mu_\infty}$. But then, by ergodicity (see Section \ref{SectionMarkov} below), its empirical measure   should converge to the unique equilibrium of  $L^{\mu_\infty}$, which is $\Pi\po \mu_\infty\pf$. Therefore, a limit of $\mu_t$ should necessarily be a fixed point of $\pi$.

More precisely, let $Lim\po \mu\pf $ be the limit set of $(\mu_t)_{t\geq 0}$, namely the set of (weak) limits of convergent sequences $\po \mu_{t_k}\pf_{k\in\mathbb N}$ when  $t_k\rightarrow \infty$. Then the following holds:
\begin{thm}\label{ThmFix}
Under Assumptions \ref{HypoUnifnu} and \ref{HypoContiNu}, almost surely,  $Lim\po \mu\pf $ is a compact connected subset of
\begin{eqnarray*}
Fix\po \pi \pf &:=& \left\{ \nu\in \mathcal P\po\mathcal M\pf,\ \nu = \pi(\nu)\right\}.
\end{eqnarray*}
\end{thm}
\noindent\textbf{Remarks.}
\begin{itemize}
\item In particular, if $Fix(\pi)$ is constituted of isolated points, then $\mu$ converges almost surely.
\item A law $m \in Fix(\pi)$ admits a positive density (still denoted by $m$) with respect to the Lebesgue measure  which 
 is also an equilibrium of the Mc-Kean Vlasov equation
\begin{eqnarray}\label{EqMCKean}
\partial_t m_t & = & \nabla_x\cdot \po - m_t \nabla_x \ln \pi({m_t}) + \nabla_x m_t \pf. 
\end{eqnarray}
This deterministic flow on $\mathcal P\po \mathcal M\pf$ describe the evolution of the law of a diffusion process whose drift depends on its law. This is a mean-field interaction. For more consideration about the link between mean-field and self-interaction, we refer to \cite{Benaim2015}.
\end{itemize}

\bigskip

In large times, due to the factor  $t^{-1}$, $\mu$ evolves slowly which, in view of Assumption \ref{HypoContiNu}, means that the dynamics evolves slowly. Hence, for $t$ and $T$ large enough, by ergodicity, the empirical law $\frac1T\int_t^{t+T} \delta_X$ should be more or less $\pi(\mu_t)$ so that,  on average, $\partial_t \po \mu_t\pf  \simeq t^{-1}\po \pi\po \mu_t\pf - \mu_t\pf$, or $\partial_t \po \mu_{e^t}\pf \simeq  \pi\po \mu_{e^t}\pf - \mu_{e^t} $. 
It will be proven in Lemma \ref{LemPK} below that 
\begin{eqnarray*}
F(\nu) &=& \pi(\nu)-\nu 
\end{eqnarray*}
is a Lipschitz map on $\mathcal P(\mathcal M)$ with respect to the total variation metric, so that it induces a continuous flow $\Psi$ on $\mathcal P(\mathcal M)$, solution of
\[\Psi_0(\nu) = \nu,\hspace{25pt}\partial_t \Psi_t(\nu) = F\po\Psi_t(\nu)\pf.\]
Our informal reasoning suggests that, in large times, the trajectory of $\mu$ should be a perturbation of the flow $\Psi$. In particular, a possible limit of $\mu$ is necessarily an equilibrium of $\Psi$. Nevertheless, because of randomness, when $\mu$ approaches an unstable equilibrium of $\Psi$, it seems unlikely that it stays in its basin of attraction, and the probability to converge to these equilibrium should be zero. Theorems \ref{ThmCVsink} and \ref{ThmPasCVsaddle} below are just a rigorous statement of these ideas. In order to retrieve the settings of \cite{BenaimRaimond}, we will restrict this study to interactions given by a symmetric potential interaction:

\begin{hyp}\label{HypoSymmetric}
There exists a smooth function $W:\mathcal M \times \mathcal M \rightarrow \mathbb R$, symmetric ($W(x,u)=W(u,x)$) such that, for all $\nu\in\mathcal P \po\mathcal M\pf$, the density of $\pi(\nu)$ is proportional to $\exp(-V_\nu)$ with
\begin{eqnarray}\label{EqVnuSym}
V_\nu(x) & =& \int W(x,u) \nu\po\dd u\pf.
\end{eqnarray}
\end{hyp}

As shown in \cite[Section 2.2]{BenaimRaimond} (see \cite[Proposition 2.9]{BenaimRaimond} for details and proofs of the following assertions), under Assumption \ref{HypoSymmetric}, the nature (stable or unstable) of the equilibria of $\Psi$ can be related to the free energy 
\begin{eqnarray*}
J(g) &:=& \frac12 \int W(x,u) g(x)g(u)\dd x\dd u + \int g(x) \ln g(x) \dd x,
\end{eqnarray*}
defined for $g\in \mathcal B_1^+ = \{f\in \mathcal C^0\po\mathcal M\pf,\ f>0,\ \int f = 1\}$. 
Indeed, $Fix\po\pi\pf$ is exactly the set of probability laws with a a density $g\in  \mathcal B_1^+ $ which is a critical point for $J$. For such a $g$, $\mathcal B_0 = \{f\in \mathcal C^0\po\mathcal M\pf,\ \int f = 0\}$ admits a direct sum decomposition
\begin{eqnarray*}
\mathcal B_0& = & \mathcal B_0^u(g)\oplus\mathcal B_0^c(g)\oplus\mathcal B_0^s(g)
\end{eqnarray*}
such that the Hessian $\mathcal D^2J(g)$ is definite negative (resp. null, resp. definite positive) on $\mathcal B_0^u(g)$ (resp. $\mathcal B_0^c(g)$, resp. $\mathcal B_0^s(g)$). The dimensions of $\mathcal B_0^u(g)$ and $\mathcal B_0^c(g)$ are finite. We say that $\nu   \in Fix\po\pi\pf$ is a non-degenerated fixed point of $\pi$ if its density $g$ is such that  $\mathcal B_0^c(g)=\{0\}$, and in that case we say it is a sink (resp. a saddle) of $\Psi$ if $\mathcal B_0^u(g)=\{0\}$ (resp. $\neq \{0\}$).

\begin{thm}\label{ThmCVsink}
Under Assumptions \ref{HypoUnifnu}, \ref{HypoContiNu} and \ref{HypoSymmetric}, let $\nu$ be a sink of $\Psi$. Then
\begin{eqnarray*}
\mathbb P\po \mu_s \underset{s\rightarrow\infty}\longrightarrow \nu \pf & > & 0.
\end{eqnarray*}
\end{thm}

To treat the case of unstable equilibria, we will add an assumption on the interaction potential $W$. We say that a symmetric, continuous function $K: \mathcal M \times\mathcal M \rightarrow \R$ is a Mercer kernel if, for all $f\in L^2\po \mathcal M,\dd x\pf$,
\begin{eqnarray*}
\int K(x,u) f(x) f(u) \dd x\dd u &\geqslant & 0.
\end{eqnarray*}
We refer to \cite[Section 2.3]{BenaimRaimond} for many examples of such kernels, among which we only recall the following: if $C$ is a metric space endowed with a probability measure $\nu$, and $G:\mathcal M\times C \rightarrow \R$ is a continuous function, then
\begin{eqnarray}\label{EqKMercer}
K(x,r) &=& \int_C G(x,u) G(r,u)\nu(\dd u)
\end{eqnarray}
is a Mercer kernel.

\begin{hyp}\label{HypoMercer}
 Assumption \ref{HypoSymmetric} holds with $W = W_+ - W_-$, where both $W_+$ and $W_-$ are Mercer kernels.
\end{hyp}

\begin{thm}\label{ThmPasCVsaddle}
Under Assumptions \ref{HypoUnifnu}, \ref{HypoContiNu}, \ref{HypoSymmetric} and \ref{HypoMercer}, let $\nu$ be a saddle of $\Psi$. Then
\begin{eqnarray*}
\mathbb P\po \mu_s \underset{s\rightarrow\infty}\longrightarrow \nu \pf & = & 0.
\end{eqnarray*}
\end{thm}

These three results are not surprising, since they are exactly similar to those of Bena\"im, Raimond and Ledoux on the self-interacting diffusion \eqref{EqDefDiffuInter}. Moreover, the structure of the proofs are very similar. The differences (and the difficulties specific to our study) are, in a sense, mostly technical, and come from the fact that the process under scrutiny, instead of being an elliptic reversible diffusion with nice regularization properties, is a kinetic piecewise deterministic Markov process, with an hybrid dynamic combining continuous time, continuous space, continuous moves and discrete jumps.

Still, the proofs of these three theorem follow so closely the works \cite{BenaimRaimond,BenaimLedouxRaimond} that, instead of recopying here large segments of the latters for completeness, we made the choice to refer to them as much as possible when the arguments can be straightforwardly adapted to our case, as long as it does not alter much the clarity of the whole presentation. That way, the present paper focuses on what is really different for the SIVJP, which drastically simplifies the presentation, as many definitions and notations are no more needed. To ease the switching from one work to the other, we tried to keep the same notations.

\subsection{Phase transitions for a toy model}

Once the above theoretical results are established, in a second part we turn to the study of a particular one dimensional case, which is the (circular integrated) telegraph process (\cite{FontbonaGuerinMalrieu2016,Volte-Face,MonmarcheRTP}). Denoting $dist_{\mathbb T}(x,z) = |e^{i x} - e^{iz}|$ for $x,z$ in the one dimensional torus $\mathbb T = \R/2\pi\Z$, we consider a quadratic interaction potential
\[\tilde W (x,z) \ =\ \rho \po \frac12 dist_{\mathbb T}^2(x,z) - 1\pf = -\rho \cos(x-z)\]
for some $\rho\in\R$,
\[W(x,z) \ = \ U(x) + \tilde W(x,y) + U(z)\]
for some smooth potential $U$ on $\mathbb T$ and, for $\nu \in \mathcal P\po\mathbb T\pf$,
\[V_\nu(x) \ = \ \int W(x,u) \nu(\dd u).\]
The self-interacting telegraph process (SITP) with exterior potential $U$, quadratic interaction and parameters $\rho\in\R$ and $\lambda_{min}>0$ is then the SIVJP on $\EE = \mathbb T \times \{-1,1\}$ with generator
\begin{eqnarray*}
L^\nu f(x,y) & =& y\partial_x f(x,y) + \po \lambda_{min} + \po y\partial_x V_\nu(x)\pf_+ \pf \po f(x,-y) - f(x,y)\pf,
\end{eqnarray*}
where $(t)_+ = \max(0,t)$ denotes the positive part. As established in \cite[Section 1.2]{MonmarcheRTP}, the invariant measure of $L^\nu$ is proportional to $\exp(-V_\nu)\otimes \po \delta_1 + \delta_{-1}\pf$, so that $\pi(\nu)$ is proportional to $\exp(-V_\nu)$. Note that, in view of \eqref{EqKMercer}, and regardless of the sign of $\rho$, the additional Assumption~\ref{HypoMercer} of Theorem~\ref{ThmPasCVsaddle} is always satisfied.  We will establish the following:

\begin{thm}\label{ThmWQuadra}
Let $(Z,\mu)$ be a SITP with  exterior potential $U$, quadratic interaction and parameters $\rho\in\R$ and $\lambda_{min}>0$. For $(a,b)$ in the unitary disk, define $\overline{\pi}_\rho(a,b)\in\mathcal P\po \mathbb T\pf$   by
\begin{eqnarray*}
\overline{\pi}_\rho(a,b)(\dd z) & = & \frac{e^{-U(z) + \rho\po a\cos(z)+b\sin(z)\pf}}{\int e^{-U(x) + \rho\po a\cos(x)+b\sin(x)\pf} \dd x}\dd z.
\end{eqnarray*}
\begin{enumerate}
\item If $U=0$, then
\begin{enumerate}[label=(\roman*)]
\item If $\rho \leqslant  2$ then $\mu_t$ almost surely converges to the Lebesgue measure on $\mathbb T$.
\item If $\rho > 2$ then there exists a deterministic $r(\rho)>0$ and a random variable $\Theta\in\mathbb T$ such that $\mu_t$ almost surely converges to $\overline{\pi}_\rho\po r\cos\Theta,r\sin\Theta\pf$.
\end{enumerate}
\item If $U(z) = -\cos(2z)$, let $\rho_c:= \po \int \cos^2 \dd \overline{\pi}_\rho(0,0)\pf^{-1}$.
\begin{enumerate}[label=(\roman*)]
\item If $\rho \leqslant  \rho_c$, then $\mu_t$ almost surely converges to $\overline{\pi}_\rho(0,0)$.
\item If $\rho > \rho_c$, then there exists a deterministic $a_*(\rho)>0$ and a random variable $\kappa\in\{-1,1\}$ (with positive probability to be 1 and to be -1) such that $\mu_t$ almost surely converges to $\overline{\pi}_\rho(\kappa a_*,0)$.
\end{enumerate}
\item If $U$ admits a non-degenerated local minimum at a point $x_0\in\mathbb T$, then for all $\delta>0$, there exist $\rho_0>0$ such that $\rho>\rho_0$ implies
\begin{eqnarray*}
\mathbb P\po \underset{t\rightarrow\infty}\limsup \po \int dist_{\mathbb T}^2(z,x_0) \mu_t\po \dd z\pf \pf<\delta \pf & > & 0.
\end{eqnarray*}
\end{enumerate}
\end{thm}

\subsection{Organization of the paper}

More examples are given in Section \ref{SectionExemples}. In Section \ref{SectionMarkov} are gathered some results on the Markovian velocity jump process without interaction. Theorem \ref{ThmFix}, \ref{ThmCVsink} and \ref{ThmPasCVsaddle} are respectively proved in Section \ref{SectionODE1}, \ref{SectionODE2} and \ref{SectionODE3}. The case of the quadratic interaction in dimension 1 is adressed in Section \ref{SectionQuadra}, in which the different points of Theorem \ref{ThmWQuadra} are proved.

\section{Examples}\label{SectionExemples}

We will present several examples for which the invariant measure of $L^\nu$ is of the product form $\Pi(\nu) \propto \exp\po-V_{\nu}\pf \otimes q$, where $q$ is invariant by rotation and $V_\nu$ is a smooth function on $\mathcal M$. In other words, for these processes, at equilibrium, the position and the velocity are independent. This is absolutely not true in a general case of modelling (see, for instance, the work of Calvez, Raoul and Schmeiser \cite{Calvez}). However, in the case of stochastic algorithms, the measure $\exp\po-V_{\nu}\pf$ is a fixed target, and the processes are especially tuned so that $\Pi(\nu)$ be of this form. In these cases, however, $\nu \mapsto V_{\nu}$ is not of the form \eqref{EqVnuSym}, namely it is not given by an interaction potential. Hence, before presenting the dynamics that will allow to sample the target measure $\Pi(\nu)$, let us explain what is $V_\nu$ is the case of adaptive algorithms.

\subsection{Adaptive algorithms}\label{SectionAdaptive}

Adaptive Biasing Force (ABF) or Potential (ABP) algorithms have been introduced in \cite{DarvePohorille,ChipotHenin}, in the context of molecular dynamics. We will only present these methods in a simplified framework, and refer to \cite{LelievreABF,BenaimBrehier2016} for a more complete introduction and discussion. The initial problem is to sample a Gibbs law with a given potential $U: \mathbb T^d \rightarrow \R$ and inverse temperature $\beta>0$, namely to construct a process $(X_t)_{t\geqslant0}$ such that, for a given observable $f$,
\begin{eqnarray}\label{EqErgoABF}
\frac1t\int_0^t f(X_s)\dd s & \underset{t\rightarrow\infty} \longrightarrow & \frac{\int f(x) e^{-\beta U(x)} \dd x}{\int e^{-\beta U(x)}\dd x}.
\end{eqnarray}
Since $U$ typically admits several local minima separated by energy barriers, a Markov process which samples the Gibbs law will typically be metastable: the transitions from one minimum to another will be rare events (especially at low temperature, namely when $\beta$ is large). This implies that the convergence \eqref{EqErgoABF} is very slow.

For the sake of simplicity, suppose that the system is represented in such a way that the first coordinate $X^1$ of $X$ is a slow variable, while the other coordinates are fast variables. In other words, the metastability of the process is mainly due to the metastability of $X^1$, which is called a reaction coordinate (or a collective variable). In practice, it is not obvious to chose good reaction coordinates, but we won't deal with this question here. We call
\begin{eqnarray}\label{EqDefA}
A(x_1) & = & - \frac1\beta\ln \int e^{-\beta U(x)} \dd x_2,\dots,\dd x_d 
\end{eqnarray}
the free energy at $X^1 = x_1$. It is, in a sense, an effective potential when we only observe the reaction coordinate: indeed, if $X=(X^1,X^2,\dots,X^d)$ is a r.v. with law proportional to $\exp\po - \beta U\pf$, then the law of $X^1$ is proportional to $\exp(-\beta A)$. Then, to reduce the metastability of the process, one can sample the law proportional to $\exp\po - \beta(U-A)\pf$, and then correct the bias by adding an exponential weight in \eqref{EqErgoABF}:
\begin{eqnarray}\label{EqDebiaise} 
\frac{\int_0^t f(X_s) e^{-\beta A(X_s^1)}\dd s}{\int_0^t   e^{-\beta A(X_s^1)}\dd s} & \underset{t\rightarrow\infty} \longrightarrow & \frac{\int f(x) e^{-A(x_1)} e^{-\beta (U-A)(x)} \dd x}{\int e^{-\beta A(x_1)}e^{-\beta (U-A)(x)}\dd x} \ = \ \frac{\int f(x) e^{-\beta U(x)} \dd x}{\int e^{-\beta U(x)}\dd x}.
\end{eqnarray}
We say that the metastability is reduced for the following reason: if $X$ is a r.v. with law $\exp\po-\beta (U-A)\pf$, then $X^1$ is   uniform on $\T$, which means it is not metastable (there are no energy barriers. In other words, the first coordinate is forced to explore all its different levels).

However, it is not possible to sample the Gibbs law with potential $U-A$, since it is not possible to compute $A$. Indeed, \eqref{EqDefA} implies an integration over a space of dimension $d-1$, where $d$ is typically large. The idea of adaptive algorithms is to learn $A$ on the fly, by calculating the integral in \eqref{EqDefA} through an MCMC method, namely thanks to the trajectory of a stochastic process $X$, and to use, simultaneously, this computation to defined biased dynamics for $X$. There are different algorithms to do so, among which we will only present the ABP and ABF ones, which enters the framework of the present paper.

The difference between ABP and ABF method is that, in the first one, the target is $A$, while in the second, the target is $ \partial_{x_1}  A$. In ABP, typically, we set
\begin{eqnarray*}
A_\nu(z) & =& -\frac1\beta\ln \int K(z,x_1) \nu(\dd x),
\end{eqnarray*}
where $K$ is a smooth approximate identity, say $K(z,x) = (\sqrt{2\pi} \varepsilon)^{-1} \exp\po-\frac{1}{2\varepsilon}|z-x|^2\pf$ for a small $\varepsilon>0$. In ABF we use that the gradient of the genuine free energy $A$ is
\begin{eqnarray*}
 \partial_{x_1} A(x_1) & = &  \frac{\int \partial_{x_1} U(x) e^{-\beta U(x)} \dd x_2,\dots,\dd x_d  }{\int e^{-\beta U(x)} \dd x_2,\dots,\dd x_d }
\end{eqnarray*}
and set
\begin{eqnarray*}
A_\nu(z) & =& \underset{f\in H^1(\mathbb T),\ \int f = 0}{\text{argmin}} \int |\partial_{x_1} U(x) - \partial_{x_1}f(z)|^2 K(z,x_1) \nu(\dd x).
\end{eqnarray*}
In both the ABP and ABF cases, set $V_\nu(x) = U(x)- A_\nu(x_1)$. As noted in \cite{LelievreABF} (in the case of a mean-field interaction rather than a self-interaction), since the bias only concerns the first variable $X^1$, the conditional laws $\mathcal L(X^2,\dots,X^d|X^1 = x_1)$ when $X$ follows $\exp(-\beta V_{\nu})$  do not depend on $\nu$. From this, it is not difficult to see that, in the ideal case where $K(z,\cdot)$ is a Dirac mass at $z$, the unique fixed point $\nu_{*}$ of $\nu \mapsto C_{\nu}^{-1} \exp\po- \beta V_\nu\pf$ (where $C_\nu$ is the normalisation constant) is proportional to $\exp\po -\beta(U-A)\pf$. In other words, $\nu_*$ is such that $A_{\nu_*} = A$. As a consequence of Theorem \ref{ThmFix}, a SIVJP such that $\pi(\nu)\propto \exp\po -\beta V_{\nu}\pf$ will be such that $A_{\mu_t}$ converges to $A$ as $t$ goes to infinity.

This is rigorously proven in \cite{BenaimBrehier2017,EhrlacherLelievreMonmarche2018}, with a slight modification: the potential used in the dynamics is $A_{\tilde \mu_t}$ where $\tilde \mu_t$ is not the occupation measure of the process, but an unbiased occupation measure
\begin{eqnarray*}
\tilde \mu_t & =& \frac{\int_0^t \delta_{X_s} e^{-\beta A_{\nu_s}(X_s)} \dd s}{\int_0^t e^{-\beta A_{\nu_s}(X_s)}  \dd s }
\end{eqnarray*}
such as used in \eqref{EqDebiaise}. This makes the proofs simpler, since in that case, everything works as if $\pi(\nu) \propto \exp(-\beta U)$ for all $\nu$. 
Nevertheless, it is not clear whether this unbiasing is useful in practice. 
Yet, to study the non-unbiased case, some difficulties arise when the kernel $K$ is not a Dirac mass, which is necessarily the case in order for $A_{\mu_t}$ to make sense since $\mu_t$ is a singular probability (at least when the dimension of the reaction coordinate is greater than one). The question to prove the convergence of the algorithm (and to characterise its limit) in the case of the non-unbiased occupation measure $\mu_t$ is a current topic of research for Br\'ehier, Bena\"im and the author.

\subsection{More general one-dimensional processes}

We consider the generator on $\mathbb T \times \R$
\begin{eqnarray*}
Lf(x,y) & =& y\partial_x f(x,y) + \lambda^\nu(x,y) \po f(x,-y) - f(x,y)\pf + r\int \po f(x,v) - f(x,y)\pf p(\dd v),
\end{eqnarray*}
where $p$ is an even function and $r>0$. The first jump term, with rate $\lambda^\nu(x,y)$, is called a bounce: depending on its environment, the process decides to turn back, without changing its scalar velocity. The second term, at rate $r$, is a refreshment: independently from its environment, the process choses a whole new velocity. If we suppose that the typical distance covered by the process between two bounces should be independent from the scalar velocity $|y|$, then  the bounce rate should be of the form $\lambda^\nu(x,y) = |y| \lambda_{sign(y)}^\nu(x)$ where $\lambda_{+}^\nu$ and $\lambda_{-}^\nu$ are two different rate of jumps. In that case,
\begin{eqnarray*}
V_{\nu}(x) & :=& \int_0^x \po \lambda_+^\nu(z) - \lambda_-^\nu(z)\pf dz
\end{eqnarray*}
is such that $y\partial_x V_\nu (x) = \lambda^\nu(x,y) - \lambda_\nu(x,-y)$. This implies that $\Pi(\nu) \propto \exp(-V_\nu) \otimes p$, as it can be checked with an integration by parts that for all smooth $f$,
\begin{eqnarray*}
 \int L^\nu f(x,y) e^{-V_\nu(x)}\dd x p(\dd v) & = & 0.
\end{eqnarray*}
Then, Assumptions \ref{HypoUnifnu} and \ref{HypoContiNu} hold as soon as the support of $p$ is compact, and
\begin{eqnarray*}
\left\|\lambda_+^{\nu_1} - \lambda_+^{\nu_2}\right\|_\infty + \left\|\lambda_-^{\nu_1} - \lambda_-^{\nu_2}\right\|_\infty & \leqslant & C  d_{TV} \po \nu_1, \nu_2\pf
\end{eqnarray*}
for some $C>0$.

\subsection{First example on $\mathbb T^d$}\label{SectionExempleTore}

For a given $\nu\in \mathcal P(\mathbb T^d) \mapsto V_{\nu} \in \mathcal C^1(\mathbb T^ d)$ and a rotation-invariant $q\in\mathcal P\po \R^d\pf$, let us describe $L^\nu$ the generator of a velocity jump process such that $\Pi(\nu) \propto  \exp\po-V_{\nu}\pf \otimes q$, and such that Assumptions \ref{HypoUnifnu} and \ref{HypoContiNu} hold. The condition on the equilibrium of $L^\nu$ is equivalent to say that for any $f\in \mathcal C^\infty\po\T^d\times \mathbb R^d\pf$,
\begin{eqnarray*}
\int L^\nu f(x,y) e^{-V_\nu(x)} \dd x q(\dd y) & =& 0.
\end{eqnarray*} 
After an integration by parts, this equivalent to say that $\lambda^\nu$ and $h^\nu$ are (weak) solutions of
\begin{eqnarray}\label{EqInvariancePi}
 \po \po y\cdot \na V_\nu(x) \pf - \lambda^\nu(x,y)\pf q(y) + \int \lambda(x,v) h^\nu(x,v,y) q(\dd v)& =  & 0. 
\end{eqnarray}
There are many possibilities. For instance, let
\begin{eqnarray*}
\lambda^\nu_1(x,y) & =& |y| |\na V_\nu(x)| + \po y\cdot \na V_\nu(x)\pf,
\end{eqnarray*}
and let $H^\nu_1(x,y,\cdot)$ be the uniform law on the sphere of radius $|y|$, in other words
\begin{eqnarray*}
H^\nu_1 f(x,y)  &=& \int_{\mathbb S^{d-1}} f(x,|y|\theta) \dd \theta.
\end{eqnarray*}
Then $\lambda^\nu_1 \geqslant 0$ and, in a weak sense,
\begin{eqnarray*}
\int v\cdot \na V_\nu(x) h^\nu_1(x,v,y) q(\dd v) & = & 0
\end{eqnarray*}
and
\[\int  |v| |\na V_\nu(x)| h^\nu_1(x,v,y) q(\dd v) \ = \   |y| |\na V_\nu(x)| \int h^\nu_1(x,v,y) q(\dd v) \ =\ q(y),\]
where we used that $h^\nu(x,-v,y)=h^\nu(x,v,y)$, and twice that $q$ is rotation invariant. This means that \eqref{EqInvariancePi} holds. More generally, this would still hold true if we had considered 
\begin{eqnarray*}
\lambda^\nu_1(x,y) & =& a(x,|y|) + \po y\cdot \na V_\nu(x)\pf,
\end{eqnarray*}
with an arbitrary function $a$ such that $a(x,|y|) \geqslant |y\cdot \na V_\nu(x)|$ for all $x,y$.

 On the other hand, if we consider the generator
\begin{eqnarray*}
L_1^\nu f(x,y) & =& y\cdot \na_x f(x,y) + \lambda_1^\nu(x,y)\po H^\nu_1 f(x,y) - f(x,y)\pf,
\end{eqnarray*}
the scalar velocity of the associated Markov process (i.e. $|Y|$) is constant (since it is fixed by both the free transport and the jumps).  As a consequence, except if $q$ is the uniform law on a sphere, the process is not ergodic with respect to $\Pi(\nu)$. But then, as in \cite{MonmarcheRTP,Doucet2015}, we can add a second jump mechanism, namely set
\begin{eqnarray*}
L^\nu f(x,y) & =& L_1^\nu f(x,y) + \bar \lambda \po \int f(x,v) q(\dd v) - f(x,y)\pf
\end{eqnarray*}
for a fixed $\bar \lambda$ (which obviously leaves $\Pi(\nu)$ invariant), in other words set
\begin{eqnarray*}
\lambda^\nu(x,y) & =& \lambda_1^\nu(x,y) + \bar\lambda\\
H^\nu(x,y) & =& \frac{\lambda_1^\nu(x,y) }{\lambda_1^\nu(x,y) + \bar\lambda} H_1^\nu + \frac{\bar\lambda}{\lambda_1^\nu(x,y) + \bar\lambda} q.
\end{eqnarray*}
Now, suppose that $|\na V_\nu |$ is bounded uniformly over $\nu \in \mathcal P\po \mathbb T^d\pf$ (which is true for an interaction potential of the form \eqref{EqVnuSym} or such as described in Section \ref{SectionAdaptive}), and that the support of $q$ is compact (for instance, $q$ could be the uniform law on a sphere, a ball, a ring, or a truncated Gaussian law). Then Assumption \ref{HypoUnifnu} holds.

Suppose, moreover, that there exists $C>0$ such that for all $\nu_1$, $\nu_2\in \mathcal P(\mathbb T^d)$, 
\[\|\na V_{\nu_1} - \na V_{\nu_2}\|_\infty \leqslant C d_{TV}(\nu_1,\nu_2)\]
 (again, this is clear in the case \eqref{EqVnuSym} and in the ABP case presented in Section \ref{SectionAdaptive}, while the ABF case requires a little work, see \cite{EhrlacherLelievreMonmarche2018}). Then Assumption \ref{HypoContiNu} holds (note that $H_1^{\nu_1} = H^{\nu_2}_1$), and  Theorem \ref{ThmFix} holds.

\bigskip

In the case of attractive interaction, the arguments of Section \ref{Multiwell} may be straightforwardly adapted to the multi-dimensional settings (see also \cite{TugautSmallNoise}). This means that, in that case, if 
\begin{eqnarray*}
V_\nu(x) & =&  \int \po U(x) + \rho  W(x,z) + U(z)\pf \nu(\dd z)
\end{eqnarray*}
with $z=x$ being the unique global minimum of $z\mapsto W(z,x)$, then for any non-degenerate local minimum of the external potential $U$, for all $\delta>0$, there exist $\rho_0>0$ such that $\rho>\rho_0$ implies
\begin{eqnarray*}
\mathbb P\po \underset{t\rightarrow\infty}\limsup \po \int |z-x_0|^2 \mu_t\po \dd z\pf \pf<\delta \pf & > & 0.
\end{eqnarray*}

\subsection{Second example on $\mathbb T^d$}

As we noted, equation \eqref{EqInvariancePi} admits many solutions, and the ones proposed above are different from the process studied in \cite{PetersdeWith,MonmarcheRTP,Doucet2015}. Let us now briefly recall the definition of the latter, explain why it does not enter our framework and discuss the difficulties which arise from its study.

Set
\begin{eqnarray*}
\lambda^\nu_1 (x,y) & =& \po y\cdot \na U(x)\pf_+\\
H_1^\nu f(x,y) & =& f\po x, R_\nu(x,y)\pf
\end{eqnarray*}
where
\begin{eqnarray*}
 R_\nu(x,y) & =& y - 2 \frac{y\na V_\nu(x)}{|\na V_\nu(x)|^2} \na V_\nu (x)
\end{eqnarray*}
is the  reflection of $y$ with respect to $V_\nu(x)$. Then, \eqref{EqInvariancePi} holds (see \cite[Section 1.4]{MonmarcheRTP}). As in the previous section, the scalar velocity is unchanged when the process jumps, but then another jump mechanism can be added. Finally, we set
\begin{eqnarray*}
\lambda^\nu(x,y) & =& \lambda_1^\nu(x,y) + \bar \lambda \\
H^\nu(x,y) & =& \frac{\lambda_1^\nu(x,y) }{\lambda_1^\nu(x,y) + \bar\lambda} H_1^\nu + \frac{\bar\lambda}{\lambda_1^\nu(x,y) + \bar\lambda} q
\end{eqnarray*}
for some constant $\bar\lambda>0$. With this definition, $\exp(-V_\nu)\otimes q$ is invariant for $L^\nu$ and, under the same assumptions on $\nu\mapsto V_\nu$ that in the previous section, Assumption \ref{HypoUnifnu} holds, and
\begin{eqnarray*}
\|\lambda^{\nu_1} - \lambda^{\nu_2} \|_\infty & \leqslant & C d_{TV}(\nu_1,\nu_2)
\end{eqnarray*}
for some $C>0$. Nevertheless, the second part of Assumption \ref{HypoContiNu} does not hold. Indeed, note that $H_1^\nu$ is a Dirac measure, so that 
\[d_{TV}\po H^{\nu_1}_1(x,y) , H^{\nu_2}_1(x,y)\pf \ = \ 1\]
as soon as $R_{\nu_1}(x,y) \neq R_{\nu_2}(x,y)$, which may be true for $d_{TV}(\nu_1,\nu_2)$ arbitrarily small.

To avoid this problem, we could try to work with a different metric on $\mathcal P\po \T^d\pf$, for instance a Wasserstein one. Recall that the $\mathcal W_1$ Wasserstein distance is
\begin{eqnarray*}
\mathcal W_1\po\nu_1,\nu_2\pf &:=& \inf\left\{ \mathbb E\po  |V_1- V_2|\pf,\ Law(V_i) = \nu_i,\ i=1,2\right\}.
\end{eqnarray*}

We may establish a bound
\begin{eqnarray*}
| R_{\nu_1}(x,y) - R_{\nu_2}(x,y) | & \leqslant & \frac{C}{|\na V_{\nu_1}|\wedge |\na V_{\nu_2}|}  \mathcal W_1 (\nu_1,\nu_2).
\end{eqnarray*}
The term $|\na V_{\nu_1}|\wedge |\na V_{\nu_2}|$ is not necessarily a real problem. This means we may have something like
\begin{eqnarray*}
\left|\po \lambda_1^{\nu_1}\wedge\lambda_1^{\nu_2}\pf \po H_1^{\nu_1} f  -  H_1^{\nu_2} f \pf\right| & \leqslant &  C \mathcal W_1 (\nu_1,\nu_2) \| \na_y f\|_\infty.
\end{eqnarray*}
Is this estimate useful ? If we try to follow the proof of Theorem \ref{ThmFix}, in the proof of Lemma~\ref{LemPK}, when coupling two processes (respectively associated with $L^{\nu_1}$ and $L^{\nu_2}$), at a jump time, even if both jump simultaneously, the new velocities are slightly different. Thus, even if the coupling is still a success at some time $t$, the two processes have drifted away one from the other. Hence, we may obtain a bound of the form 
\begin{eqnarray*}
\|P_t^{\nu_1} f - P_t^{\nu_2} f\|_\infty & \leqslant &  C(t) \po \| f\|_\infty +\|\nabla f\|_\infty\pf \mathcal W_1 (\nu_1,\nu_2).
\end{eqnarray*}
for some locally finite function $C$, which is not a problem by itself. But then in the proof of Lemma \ref{LemQ} (as it is for now) we would need to control $\|\nabla Q^\nu f\|_\infty$, where $Q^\nu f = \int_0^\infty P_t f\dd t$. It is unclear whether it is possible, since the semi-group has no regularisation property, and is not a contraction of the Wasserstein space. 

For now, we haven't find a way to prove Theorem \ref{ThmFix} for this process.


\subsection{Isotropic interaction on the sphere}

The unit tangent bundle of the sphere $\mathbb S^{d-1}$ ($d\geqslant 2$) may be seen as
\[\EE \ =\ \left\{(x,y) \in \R^{2d}, \ |x|=|y|=1,\ y\cdot x = 0\right\}.\]
In that case, 
\begin{eqnarray*}
\varphi_t(x,y) & =& \po x \cos( t) + y\sin\po  t\pf, -x\sin(t) + y \cos(t)\pf
\end{eqnarray*}
and, for $f \in \mathcal C^\infty \po \R^{2d}\pf$,
\[D f(x,y) = \underset{t\rightarrow 0}\lim \frac{f\po \varphi_t(x,y)\pf - f(x,y)}{t}\ = \ y\cdot \na_x f(x,y) -   x\cdot \na_y f(x,y).\]
For $\nu\in\mathcal P(\mathbb S^{d-1})$, let $V_\nu(x) = \int W(x,z)\nu(\dd z)$, with a smooth interaction potential $W$ which is rotation invariant, namely $W(Rx,Ry) = W(x,y)$ for any rotation $R$ of $\mathbb S^{d-1}$. By analogy with the torus case of Section \ref{SectionExempleTore}, suppose that the jump rate only depends on $y\cdot \na V_\nu(x)$, namely that $\lambda^\nu(x,y) = \psi(y\cdot \na V_\nu)$ for some continuous, positive function $\psi$, and suppose moreover that for all $(x,y)\in\EE$, the jump kernel $H^\nu(x,y)$ is uniform over $\mathbb S^{d-2}_x=\{v\in \mathbb S^{d-1},\ x\cdot v = 0\}$. Then Assumptions \ref{HypoUnifnu} and \ref{HypoContiNu} holds.

\bigskip

Contrary to the torus case, the invariant measure of $L^\nu = D + \lambda^\nu (H^\nu - I)$ is not explicit, except for one case. Indeed, denote $\nu_0$ the uniform law over $\mathbb S^{d-1}$. Since $W$ and $\nu_0$ are rotation invariant, $x\mapsto V_{\nu_0}$ is constant, so that $y\cdot\na V_{\nu_0}(x)=0$ for all $(x,y)\in\EE$, and
\begin{eqnarray*}
L^{\nu_0} f(x,y) & =& D f(x,y) + \psi(0) \po \int_{\mathbb S^{d-2}_x} f(x,v) \dd v - f(x,y)\pf .
\end{eqnarray*}
The invariant measure of $L^{\nu_0}$ is the uniform measure over $\EE$. Indeed, the latter is left invariant both by $\varphi_t$ for all $t\geqslant 0$ and by the jump part of $L^{\nu_0}$. As a consequence, $\pi(\nu_0) = \nu_0$, in other words $\nu_0 \in Fix(\pi)$.

If $(Z_t,\mu_t)_{t\geqslant}$ is a SIVJP on $\EE$ with generator $L^\nu$, even for particular choices of $\psi$, it is not clear whether $\mathbb P\po \mu_t \rightarrow \nu_0\pf$ is 1, positive or zero; whether $Fix(\pi)$ is reduced to $\nu_0$ or not and, if it is not, if $\nu_0$ is a sink for the deterministic asymptotic flow $\Psi$ on $\mathcal P(\mathcal M)$. We leave this question for further work.


\section{Preliminary results without self-interaction}\label{SectionMarkov}

In this section, we study, for a fixed $\nu$, the Markov semi-group $(P_t^\nu)_{t\geqslant0}$ with generator
\begin{eqnarray}\label{EqGeneMarkovNu}
L^\nu f(x,y)\ &=& D f(z) + \lambda^\nu\po z\pf \po   H^\nu f(z)   - f(z)\pf.
\end{eqnarray} 

\subsection{Equilibrium}

The ergodicity of the process will be established by usual coupling arguments, very similar to \cite[Lemma 5.2]{MonmarcheRTP}. Nevertheless, as we wish to obtain estimates which are uniform over $\nu \in \mathcal P\po \mathcal M\pf$, we will use a reference dynamics which does not depend on $\nu$. Let $\lambda_{\min}$ and $c$ be given by Assumption \ref{HypoUnifnu}, which is enforced in all this section. Set $\lambda^0 = c\lambda_{\min}$,
\begin{eqnarray*}
H^0 f(x,y) & = & \int f(x,r\theta) p(\dd r)\dd \theta,\\
L^0 f(z) &=& Df(z) + \lambda^0\po H^0 f(z) - f(z)\pf,
\end{eqnarray*}
and $(P_t^0)_{t\geqslant 0}$ be the Markov semi-group associated with $L^0$. Then, we can decompose
\begin{eqnarray*}
L^\nu f(z) & = & L^0 f(z) + \tilde \lambda^\nu(z) \po \tilde H^\nu f(z) - f(z)\pf,
\end{eqnarray*}
where
\begin{eqnarray*}
\tilde \lambda^\nu(z)& = & \lambda^\nu(z) - c\lambda_{\min}\\
\tilde H^\nu f(z) & = & \frac{\lambda^\nu(z) - \lambda_{\min}}{\lambda^\nu(z) - c\lambda_{\min}} H^\nu f(z) + \frac{(1-c)\lambda_{\min}}{\lambda^\nu(z) - c\lambda_{\min}} \frac{(H^\nu - c H^0) f(z)}{1-c}.
\end{eqnarray*}
From Assumption \ref{HypoUnifnu}, $\tilde \lambda^\nu\geqslant 0$, and $(1-c)^{-1}(H^\nu - c H^0) f(z)$ (hence $\tilde H^\nu$) is a Markov operator. From this decomposition, we get the following:

\begin{lem}\label{LemP01}
For all positive $f\in L^\infty\po \EE\pf$, $\nu \in \mathcal P\po\mathcal M\pf$, $t\geqslant0$ and $z\in\EE$,
\begin{eqnarray*}
P_t^\nu f(z) & \geqslant & e^{-\lambda_{\max} t} P_t^0 f (z).
\end{eqnarray*}
\end{lem}
\begin{proof}
Let $(Z_t')_{t\geqslant 0}$ be a Markov process with generator $L^0$ and $Z_0' = z$, so that $P_t^0 f (z) = \E\po f(Z_t')\pf$. Let $E$ be a standard exponential random variable, independent from $Z'$, and
\begin{eqnarray*}
T & =& \inf \left\{  t>0,\ E < \int_0^t \tilde \lambda^\nu (Z_s') \dd s\right\}.
\end{eqnarray*}
Set $Z_t = Z_t'$ for $t<T$, and draw $Z_T$ according to the law $\tilde H(Z_T')$. For $t\geqslant T$, let $Z$ evolve according to the Markov dynamics with generator $L^\nu$, independently from $Z'$. Then $Z$ is a Markov process with generator $L^\nu$ and $Z_0 = z$, so that
\begin{eqnarray*}
P_t^\nu f(z) & = & \E \po f(Z_t)\pf\\
& = & \mathbb P \po E > \lambda_{\max} t \pf \E \po f(Z_t)\ |\ E > \lambda_{\max} t \pf + \mathbb P \po E \leqslant \lambda_{\max} t \pf \E \po f(Z_t)\ |\ E \leqslant \lambda_{\max} t \pf\\
& \geqslant & e^{-\lambda_{\max}t} \E \po f(Z_t)\ |\ E > \lambda_{\max} t \pf.
\end{eqnarray*}
Note that $E > \lambda_{\max} t $ implies that $t<T$, hence $Z_t = Z_t'$. Since $E$ is independent from $Z_t'$, $\E \po f(Z_t)\ |\ E > \lambda_{\max} t \pf = P_t^0 f(z)$.
\end{proof}

The next step is a Doeblin minoration condition for $P^0$:

\begin{lem}\label{LemP02}
There exist $t_0,\kappa>0$ such that for all positive $f\in L^\infty\po \EE\pf$,
\begin{eqnarray*}
P_{t_0}^0 f(z) & \geqslant & \kappa \int f(u,v) p(\dd r)\dd \theta \dd u \ := \ \kappa\int f(z) m_0(\dd z),
\end{eqnarray*}
where $\dd u$ stands for the uniform law on $\mathcal M$.
\end{lem}
\begin{proof}
Adapting the proofs of \cite[Proposition 2.2 and Theorem 6.2]{BLBMZ3}, we get that $P^0$ is a Feller semi-group and that, for $\varepsilon \in [\varepsilon_0,2\varepsilon_0]$ for a fixed small $\varepsilon_0>0$, for all $x\in\mathcal M$, there exist $\delta,\kappa >0$ such that for all positive $f\in L^\infty\po \EE\pf$, 
\begin{eqnarray*}
\int P_{\varepsilon}^0 f(x,r\theta)p(\dd r)\dd \theta & \geqslant & \kappa \int_{\mathcal B_x(\delta)}\po \int f(u,r\theta) p(\dd r)\dd \theta \pf \dd u,
\end{eqnarray*}
where $\mathcal B_x(\delta)$ is the ball in $\mathcal M$ centered at $x$ with radius $\delta$. By density, we can restrict to continuous functions $f$, from which we get that this inequality holds with $\kappa,\delta>0$ which are uniform over $x\in\mathcal M$ and $\varepsilon \in [\varepsilon_0,2\varepsilon_0]$. Hence, for $n\in \mathbb N$ large enough, there exist $\kappa_n>0$ such that, for all $x\in\mathcal M$, $\varepsilon \in [\varepsilon_0,2\varepsilon_0]$ and positive $f$,
\begin{eqnarray*}
\int P_{n\varepsilon}^0 f(x,r\theta)p(\dd r)\dd \theta & \geqslant & \kappa_n   \int f(u,r\theta) p(\dd r)\dd \theta  \dd u.
\end{eqnarray*}
For a given $z\in\EE$, consider $(Z_t)_{t\geqslant0}$ a process with generator $L^0$ and $Z_0 = z$, and $t_1$ its first jump time. Then, for $t\geqslant 2n\varepsilon_0$,
\begin{eqnarray*}
P_t^0 f(z) & =& \E \po f(Z_t)\pf\\
& \geqslant & \mathbb P \po t-2n\varepsilon_0\leqslant t_1\leqslant t - n\varepsilon_0  \pf \E \po f(Z_t) \left|t-2n\varepsilon_0\leqslant t_1\leqslant t - n\varepsilon_0 \right. \pf\\
& =& \int_{t-2n\varepsilon_0}^{t-n\varepsilon_0} \po \int P_{t-s}^0 f\po \varphi_s^{(1)}(z),r\theta\pf p(\dd r) \dd \theta\pf \lambda^{0} e^{-\lambda_0 s} \dd s\\
& \geqslant &  \kappa_n \po e^{-\lambda^0(t-2n\varepsilon)} - e^{-\lambda^0(t-n\varepsilon)}\pf  \int f(z) m_0(\dd z),
\end{eqnarray*}
which concludes
\end{proof}

\begin{prop}\label{PropErgodique}
There exist $C_1,\rho>0$ such that for all $m_1,m_2 \in \mathcal P \po \EE\pf$, $\nu\in\mathcal P\po\mathcal M\pf$ and $t\geqslant 0$,
\begin{eqnarray*}
d_{TV}\po m_1 P_t^\nu,m_2 P_t^\nu\pf & \leqslant & C_1 e^{-\rho t} d_{TV}\po m_1,m_2\pf. 
\end{eqnarray*}
\end{prop}
\begin{proof}
From Lemmas \ref{LemP01} and \ref{LemP02}, there exist $t_0,\kappa>0$ such that for all $\nu\in\mathcal P\po\mathcal M\pf$ and $z\in\EE$,
\begin{eqnarray}\label{EqDoeblinPt0}
P_{t_0}^\nu f(z) & \geqslant & \kappa \int f(z) m_0(\dd z),
\end{eqnarray}
so that $\tilde R = (1-\kappa)^{-1} (P_{t_0} - \kappa m_0)$ is a Markov operators.

Let $(Z_0^1,Z_0^2)$ be an optimal coupling of $(m_1,m_2)$, in the sense that $Z_0^i \sim m_i$ for $i=1,2$ and that
\begin{eqnarray*}
\mathbb P\po Z_0^1 \neq Z_0^2\pf & =& d_{TV}(m_1,m_2).
\end{eqnarray*}
Let $U_0$, $U_1$ and $U_2$ be random variables with respective laws $m_0$, $\delta_{Z_0^1} \widetilde R$ and $\delta_{Z_0^1} \widetilde R$. With probability $\kappa$, set $Z_1^1 = Z_1^2 = U_0$. Else (with probability $(1-\kappa)$), if $Z_0^1 = Z_0^2$, set $Z_1^1 = Z_1^2 = U_1$, and if $Z_0^1 \neq Z_0^2$, set $Z_1^1 =  U_1$ and $Z_1^2 = U_2$. Thus, $Z_1^1 \sim m_1 P^\nu_{t_0}$ and $Z_1^2 \sim m_2 P^\nu_{t_0}$, and
\[ d_{TV}\po m_1 P_{t_0}^\nu,m_2 P_{t_0}^\nu\pf  \ \leqslant \ \mathbb P\po Z_1^1 \neq Z_1^2\pf \ \leqslant \ (1-\kappa)  d_{TV}\po m_1  ,m_2  \pf . \]
On the other hand, the Markov property of $P_t^\nu$, for all $t\geqslant 0$, implies that
\[ d_{TV}\po m_1 P_{t}^\nu,m_2 P_{t}^\nu\pf  \ \leqslant \  d_{TV}\po m_1  ,m_2  \pf, \]
so that the semi-group property of $P^\nu$ yields
\begin{eqnarray*}
d_{TV}\po m_1 P_t^\nu,m_2 P_t^\nu\pf & \leqslant &  (1-\kappa)^{\lfloor t/t_0\rfloor} d_{TV}\po m_1,m_2\pf. 
\end{eqnarray*}
\end{proof}

\begin{lem}\label{LemMesureInvariante}
For all $\nu \in \mathcal P\po\mathcal M\pf$, the Markov semi-group $P^\nu$ admits a unique invariant measure $\Pi(\nu)\in\mathcal P\po\EE\pf$, whose first marginal $\pi(\nu) \in \mathcal P\po\mathcal M\pf$ admits a positive density wih respect to the Lebesgue measure on $\mathcal M$.
\end{lem}
\begin{proof}
Let $t>0$ be large enough so that, from Proposition \ref{PropErgodique}, $P_t^\nu$ is a contraction of $\mathcal P\po \EE\pf$ (which, endowed with the total variation metric is a Banach space). By the Banach fixed-point  Theorem, there exists a unique $m\in\mathcal P(\EE)$ such that $mP_t^\nu = m$ and, from the semi-group property of $P^\nu$ and the uniqueness of $m$, $mP_s^\nu P_t = m P_s^\nu$ implies that $mP_s^\nu = m$ for all $s\geqslant 0$.


\bigskip

To prove that the first marginal of this invariant measure $m$ admits a density with respect to the Lebesgue measure, following \cite[Proposition 5.9]{BLBMZ3}, we will consider the embedded chain associated to the continuous-time process. More precisely, by writing
\begin{eqnarray*}
L^\nu f &= & D f + \lambda_{\max} \po \widehat H^\nu  f - f\pf 
\end{eqnarray*}
with 
\begin{eqnarray*}
\widehat H^\nu f(x,y) & = & \frac{\lambda^\nu}{\lambda_{\max}} H^\nu f + \po 1 - \frac{\lambda^\nu}{\lambda_{\max}} \pf f,
\end{eqnarray*}
we consider an alternative construction of the process: the jumps occur at constante rate $\lambda_{\max}$, with kernel $\widehat{H}^\nu$. In other words, we add phantom jumps at rate $\lambda_{\max}-\lambda^\nu$, at which nothing happens. Let 
\begin{eqnarray*}
\tilde P^\nu f(x,y) & =& \int_0^\infty \widehat{H}^\nu f\po \varphi_s(x,y)\pf \lambda_{\max} e^{-\lambda_{\max} s} \dd s \ = \ \mathbb E \po f(Z_{T})\pf
\end{eqnarray*}
where $Z$ is a Markov process with generator $L^\nu$ with $Z_0=(x,y)$ and $T$ is its first jump time (with possible phantom jumps). If $(X,Y)$ follows the law $\tilde P^\nu(x,y)$, then $X = \varphi_{T}^{(1)}(x,y)$ where $T$ is an exponential variable with parameter $\lambda_{\max}$. In particular, if a law $n\in\mathcal P(\EE)$ admits a first marginal which admits a density with respect to the Lebesgue measure, so does $n\widetilde P^\nu$.

 Following \cite[Proposition 5.2]{BLBMZ3} (see also \cite[Theorem 34.31, p.123]{Davis}), $\tilde P^\nu$ admits a unique invariant measure $\tilde m$, and $m =  \tilde m\tilde P^\nu$. Hence, it is sufficient to prove that the first marginal of $\tilde m$ admits a density.
 
 Note that, for a positive $f$, under Assumption \ref{HypoUnifnu},
 \begin{eqnarray*}
\tilde P^\nu f(x,y) & \geqslant & c \frac{\lambda_{\min}}{\lambda_{\max}} \int_0^\infty  \int f\po \varphi_s^{(1)}(x,y),r\theta\pf p(\dd r)\dd \theta \lambda_{\max} e^{-\lambda_{\max} s} \dd s,
\end{eqnarray*}
from which it is clear that, for $k$ large enough, there exists $\kappa_k >0$ such that for all $z\in\EE$,
\begin{eqnarray*}
\po \tilde P^\nu \pf^k f(z) & \geqslant & \kappa_k \int f(z) m_0(z).
\end{eqnarray*}

For $n\in \mathcal P\po\EE\pf$, if $\mathcal L aw(X,Y) = n$, denote $n_1 = \mathcal L aw(X)$ and $n_x = \mathcal L aw(Y|X=x)$, so that $n = n_1 n_x$. Decomposing $n_1 = p n_{1,ac} + (1-p) n_{1,\perp}$ with $n_{1,ac}$ (resp. $n_{1,\perp}$) absolutely continuous (resp. singular) with respect to the Lebesgue measure on $\mathcal M$, set $n_{ac} = n_{1,ac} n_x$ and $n_{\perp} = n_{1,\perp} n_x$. We call $p$ the mass of the continuous part of $n_1$, which is uniquely defined. Then,
\[\tilde m_1 \ = \ \int \tilde m \dd y \ = \ \int \tilde m  \po \tilde P^\nu \pf^k  \dd y  \ \geqslant \  \int \po p \tilde m_{ac} \po \tilde P^\nu \pf^k  + (1-p)\kappa m_0 \pf \dd y, \]
which implies that the mass of the continuous part of $\tilde m_1$, which is $p$, is at least $p+(1-p)\kappa$, so that $p=1$.

Finally, the fact that the density of $\pi(\nu)$ is positive is a consequence of the minoration \eqref{EqDoeblinPt0} which gives 
\[\pi(\nu) \ = \ \int \Pi(\nu) \dd y \ = \ \int \Pi(\nu) P^\nu_{t_0} \dd y \ \geqslant \ \kappa \int m_0 \dd y \ = \ \kappa \dd u\]
with $\dd u$ the uniform law on $\mathcal M$.


\end{proof}


Let $K^\nu f = f - \Pi(\nu) f$ denote the orthogonal projection operator  (in $L^2\po \Pi(\nu)\pf$) on the orthogonal of the constant functions. Since
\[ P_t^\nu K^\nu f(z) \ = \ \delta_z P_t^\nu K^\nu f - \Pi(\nu) P_t^\nu K^\nu f,\]
Proposition \ref{PropErgodique} immediatly yields
\begin{lem}\label{LemCVequilibre}
For all $t\geqslant 0 $, $\nu\in \mathcal P\po \mathcal M\pf$ and $f\in L^\infty(\EE)$,
\begin{eqnarray*}
\| P_t^\nu K^\nu f \|_\infty & \leq & C_1 e^{-\rho t}\| K^\nu f \|_\infty.
\end{eqnarray*}
\end{lem}
In particular, the operator
\begin{eqnarray*}
Q^\nu f & := & -\int_0^\infty P_t^\nu K_\nu f \dd t
\end{eqnarray*}
is well-defined for $f\in L^\infty(\EE )$ and satisfies
\begin{eqnarray}\label{EqBorneQ}
\| Q^\nu f\|_\infty & \leq & \frac{C_1}{\rho} \|   f\|_\infty  .
\end{eqnarray}
When $f$ is in the domain of $L$,  $\partial_t P_t^{\nu} f = P_t^{\nu} L^\nu f = L^\nu P_t^{\nu} f$ for all $t\geq 0$, and
\[ L^\nu Q^\nu f \ = \  Q^\nu L^\nu f \ = \ -\int_{0}^\infty \partial_t P_t^{\nu} K^\nu f \dd t \ =\  K^\nu f.\]
In other words, for all $g \in L^\infty(\EE)$ with $\int g \dd \Pi(\nu) = 0$, the solution of the Poisson equation $L^\nu f = g$ is given by $f = Q^\nu g$.


\subsection{Dependency on $\nu$}

We keep the notations of the previous section. In this subsection, both Assumptions \ref{HypoUnifnu} and \ref{HypoContiNu} are enforced.

\begin{lem}\label{LemPK}
There exist $C>0$ such that for all $t\geq 0$,  $\nu_1,\nu_2\in \mathcal P\po\mathcal M\pf$ and $f\in L^\infty(\EE)$,
\begin{eqnarray*}
\|P_t^{\nu_1} f - P_t^{\nu_2} f\|_\infty & \leqslant &  C t\| f\|_\infty d_{TV}\po\nu_1,\nu_2\pf\\
\|K^{\nu_1} f - K^{\nu_2} f\|_\infty & \leqslant & C\| f\|_\infty  d_{TV}\po\nu_1,\nu_2\pf
\end{eqnarray*}
\end{lem}
\begin{proof}
For $z_1,z_2\in \EE^2$, we denote by $\mathbf H (z_1,z_2)$ the law of an optimal coupling of $H^{\nu_1}(z_1)$ and $H^{\nu_2}(z_2)$, which means that, if $(Z^1,Z^2)\sim \mathbf H(z_1,z_2)$, then $Z^i \sim H^{\nu_i}(z_i)$ for $i=1,2$ and  $\mathbb P \po Z^1 \neq Z^2 \pf = d_{TV}\po H^{\nu_1}(z_1),H^{\nu_2}(z_2)\pf$. We denote 
\[\po H^{\nu_1}\otimes I\pf f((x_1,y_1),z_2) = \int f\po(x_1,v),z_2\pf h^{\nu_1}(x_1,y_1,\dd v),\]
 similarly for $I\otimes H^{\nu_2}$, and
 \begin{eqnarray*}
 (D\otimes D ) f(z_1,z_2) & =& \underset{t\rightarrow 0}\lim \frac{f\po \varphi_t(z_1),\varphi_t(z_2)\pf}{t}.
 \end{eqnarray*}

For $z\in \EE$, let $\po Z_t,\widetilde {Z_t}\pf_{t\geqslant0}$ be the Markov process on $\EE^2$ starting at $\po Z_0,\widetilde{Z_0}\pf=(z,z)$ and with generator 
\begin{eqnarray*}
\mathbf L f  &=& (D\otimes D) f   +   \po \lambda^{\nu_1}  \wedge \lambda^{\nu_2}\pf \po \mathbf H f - f \pf \\ 
& & +   \po \lambda^{\nu_1}  -  \po \lambda^{\nu_1}  \wedge \lambda^{\nu_2}\pf\pf_+   \po \po H^{\nu_1}\otimes I\pf f - f \pf\\
& & +   \po \lambda^{\nu_2} -  \po \lambda^{\nu_1}  \wedge \lambda^{\nu_2}\pf\pf_+    \po \po I\otimes H^{\nu_2}\pf f- f \pf.
\end{eqnarray*}
As can be seen by checking that $\mathbf L f = L^{\nu_i} f$ whenever $f$ depends only on $z_i$, $i=1$ or $2$, the first marginal $Z$ (resp. the second marginal $\widetilde Z$) is a Markov process associated to $L^{\nu_1}$ (resp. $L^{\nu_2}$), so that
\begin{eqnarray*}
| P_t^{\nu_1} f(z) - P_t^{\nu_2} f(z)| & = & \left|\mathbb{E}\po f\po Z_t\pf - f\po \widetilde{Z_t}\pf\pf\right|\\
&\leqslant & \|f\|_\infty \mathbb P\po  Z_t \neq \widetilde {Z_t}\pf .
\end{eqnarray*}
The dynamic given by $\mathbf L $ ensures that, as much as possible, both processes jump simultaneously and that, when they do so, the probability that they take the same velocity is as high as possible. For the sake of simplicity, by adding phantom jumps as in the proof of Lemma \ref{LemMesureInvariante}, we write
\begin{eqnarray*}
\mathbf L f  &=& (D\otimes D) f   +   \lambda_{\max} \po \widehat{\mathbf H} f - f \pf.  
\end{eqnarray*}
Assumption \ref{HypoContiNu} implies that there exists $C>0$ such that, for all $z\in\EE$, if $(Z^1,Z^2)\sim \widehat{\mathbf H}(z,z)$, then $\mathbb P\po Z^1 \neq Z^2 \pf \leqslant C d_{TV}(\nu_1,\nu_2)$.

At time $t=0$, $Z_t  = \widetilde Z_t$, and this remains true at least up to the fist (possibly phantom) jump time $T_1$, which is an exponential r.v. with parameter $\lambda_{\max}$. At time $T_1$, independently from $T_1$, there is a probability at least $1-C d_{TV}(\nu_1,\nu_2)$ that $Z_{T_1} = \widetilde Z_{T_1}$, and thus the processes stay equal up to the next jump time $T_2$, and so on. Conditionnaly to $K_t$ the number of jumps between times 0 and $t$, the probability that the processes haven't split yet at time $t$ is greater than $\po 1 -C d_{TV}(\nu_1,\nu_2)\pf_+^{K_t}$. Since $K_t$ follows a Poisson law with parameter $\lambda_{\max} t$,
\begin{eqnarray*}
\mathbb P\po  Z_t \neq \widetilde {Z_t}\pf & \leqslant & 1 - \E\po \po 1 -C d_{TV}(\nu_1,\nu_2)\pf_+^{K_t}\pf \  \leqslant \ 1 - e^{ - C d_{TV}(\nu_1,\nu_2) \lambda_{\max} t},
\end{eqnarray*}
which concludes the proof of the first assertion.

As far as the second one is concerned, write
\begin{eqnarray*}
K^{\nu_1} f(z) - K^{\nu_2} f(z)
 & = & \Pi(\nu_1)K^{\nu_2} f\\
  & = & \Pi(\nu_1)L^{\nu_2} Q^{\nu_2}f\\
  & = & \Pi(\nu_1)\po L^{\nu_2}-L^{\nu_1}\pf Q^{\nu_2}f,
\end{eqnarray*}
where we used that, $\Pi(\nu_1)$ being invariant for $L^{\nu_1}$, $\Pi(\nu_1)L^{\nu_1} = 0$. From Assumption \ref{HypoContiNu} and the bound \eqref{EqBorneQ}, there exists $C>0$ such that
\[|\po L^{\nu_2}-L^{\nu_1}\pf Q^{\nu_2}f| \ \leqslant\ C\| Q^{\nu_2} f\|_\infty  d_{TV}\po\nu_1,\nu_2\pf \ \leqslant \ \rho^{-1} C_1 C\|   f\|_\infty  d_{TV}\po\nu_1,\nu_2\pf.\]

\end{proof}

\begin{lem}\label{LemQ}
There exists $C_2>0$ such that for all    $\nu_1,\nu_2\in \mathcal P\po\mathcal M\pf$ and $f\in L^\infty(\EE)$,
\begin{eqnarray*}
\| Q^{\nu_1} f  - Q^{\nu_2} f \|_\infty & \leqslant &   C_2\| f\|_\infty    d_{TV}\po\nu_1,\nu_2\pf.
\end{eqnarray*}
\end{lem}

\begin{proof}
Using that 
\[\int_{t}^\infty P_s^{\nu} K^{\nu} f \dd s \ =\ \int_{0}^\infty P_{t+s}^{\nu} K^{\nu} f \dd s  \ =\  P_t^\nu K^{\nu} Q^\nu f,\]
we decompose, for any $t>0$,
\begin{eqnarray*}
Q^{\nu_2} f  - Q^{\nu_1} f &=& \int_0^t \co \po P_s^{\nu_1} - P_s^{\nu_2} \pf K^{\nu_1} f  + P_s^{\nu_2}\po K^{\nu_1}-K^{\nu_2}\pf \cf f \dd s  +\ P_t^{\nu_1} K^{\nu_1}\po Q^{\nu_1}    - Q^{\nu_2} \pf f\\\
& &  +\ \po P_t^{\nu_1} -P_t^{\nu_2}\pf K^{\nu_1}  Q^{\nu_2}   f  \ +\  P_t^{\nu_2} \po  K^{\nu_1}    - K^{\nu_2} \pf Q^{\nu_2}  f.
\end{eqnarray*}
Lemmas \ref{LemCVequilibre} (with $t$ large enough so that $C_1 e^{-\rho t} < \frac12$) and \ref{LemPK} thus yield, for some $C$,
\begin{eqnarray*}
\| Q^{\nu_1} f  - Q^{\nu_2} f\|_\infty  &\leqslant &  \frac12 \| Q^{\nu_1} f  - Q^{\nu_2} f\|_\infty  + C \| f\|_\infty    d_{TV}\po\nu_1,\nu_2\pf,
\end{eqnarray*}
which concludes.
\end{proof}

\section{The limiting flow}\label{SectionODE}

\subsection{Asymptotic pseudotrajectory}\label{SectionODE1}

Let $(f_k)_{k\in\mathbb N}$ be a sequence of $\mathcal C^\infty$ functions on $\mathcal M$ which is dense in the unitary ball of $\mathcal C^0\po\mathcal M\pf$ (endowed with the uniform metric) and for $\nu_1,\nu_2\in \mathcal P\po\mathcal M\pf$ let
\begin{eqnarray*}
d_w\po \nu_1,\nu_2\pf &=& \sum_{k\in \mathbb N}\frac{1}{2^k}|\nu_1 f-\nu_2f|,
\end{eqnarray*}
which is a metric that induces the weak topology on $\mathcal P\po\mathcal M\pf$. Then a continuous function $\xi$ from $\R_+$ to $\mathcal P\po\mathcal M\pf$ is called (see \cite{BenaimHirsch}) an asymptotic pseudotrajectory for the flow $\Psi$ if  for all $T$,
\begin{eqnarray*}
\underset{0\leqslant h \leqslant T}\sup \ d_w \po \xi(t+h),\Psi_h\po \xi(t)\pf\pf  & \underset{t\rightarrow \infty}\longrightarrow & 0.
\end{eqnarray*}
\begin{prop}\label{PropPseudoTrajectoire}
Let $(Z_t,\mu_t)_{t\geq 0}$ be a SIVJP and  $\zeta_t := \mu_{e^t}$. Under Assumptions \ref{HypoUnifnu} and \ref{HypoContiNu}, $\zeta$ is an asymptotic pseudotrajectory for   $\Psi$.
\end{prop}
\begin{proof}
Following the proof of \cite[Theorem 3.6, parts (i)(b) and (ii)]{BenaimLedouxRaimond}, Proposition \ref{PropPseudoTrajectoire} ensues from Proposition \ref{PropEpsi} below.
\end{proof}

Let us first remark Theorem \ref{ThmFix} is deduced from this result:
 
\begin{proof}[Proof of Theorem \ref{ThmFix}]
The fact that Proposition \ref{PropPseudoTrajectoire} implies Theorem \ref{ThmFix} is proved in \cite[Section 4]{BenaimRaimond}. More precisely, according to \cite[Theorem 3.7]{BenaimLedouxRaimond}, the limit set of an asymptotic pseudotrajectory has the property to be \emph{attractor free} (see \cite[Section 3.3]{BenaimLedouxRaimond} for the definition), and the proof of \cite[Theorem 2.4]{BenaimRaimond} (which is exactly Theorem \ref{ThmFix}) only relies on this property and on the flow $\Psi$, the latter being exactly the same in our case than in the work of Bena\"im and Raimond.
\end{proof}

Consider a SIVJP $(Z_t,\mu_t)_{t\geq 0}=(X_t,Y_t,\mu_t)_{t\geq 0}$ and let $\zeta_t = \mu_{e^t}$. Set 
\begin{eqnarray*}
\varepsilon_t(s) & = & \int_t^{t+s} \po \delta_{X_{(e^u)}} - \pi\po \zeta_u\pf \pf \dd u \hspace{5pt}=\hspace{5pt} \int_{e^t}^{e^{t+s}} \frac{\delta_{X_{u}} - \pi\po \mu_u\pf}{u}  \dd u,
\end{eqnarray*}
which is a signed measure on $\mathcal M$. 
\begin{prop} \label{PropEpsi}
Under Assumptions \ref{HypoUnifnu} and \ref{HypoContiNu}, there exists a constant $C_3$ such that for all $f\in C^{\infty}\po \mathcal M\pf$ and $T,t,\delta >0$,
\begin{eqnarray}\label{EqPropEpsi}
\mathbb P\po\left.  \underset{0\leq s\leq T}\sup |\varepsilon_t(s) f|\geqslant \delta \ \right|\ \mathcal F_{e^t} \pf  & \leqslant &  \frac{C_3 e^{-t}}{\delta^2} \| f\|_\infty^2.
\end{eqnarray} 
\end{prop}

\begin{proof}
Let $f\in\mathcal C^\infty \po\mathcal M\pf$, which we abusively amalgamate as a function on $\EE$ by $f(x,y) := f(x)$ so that, using the notations of Section \ref{SectionMarkov}, we get
\[\varepsilon_t(s) f = \int_{e^t}^{e^{t+s}} \frac{K^{\mu_u} f(Z_u)}{u} \dd u = -\int_{e^t}^{e^{t+s}} \frac{L^{\mu_u} Q^{\mu_u}  f(Z_u)}{u} \dd u .\]
Let $F_t(z) = \frac1t Q^{\mu_t} f(z)$, and note that $z\mapsto F_t(z)$ is $\mathcal C^\infty$. On the other hand, $1<t\mapsto F_t(z)$ is Lipschitz: indeed, from Lemma \ref{LemQ},
\begin{eqnarray*}
| Q^{\mu_{t+s}} f - Q^{\mu_t} f| & \leqslant & C_2\| f\|_\infty d_{TV}\po \mu_{t+s}\  ,\ {\mu_t}\pf\\
& \leqslant & C_2\| f\|_\infty \frac{s}{r+t+s}
\end{eqnarray*}
where we used that $\mu_{t+s} = \frac{r+t}{r+t+s} \mu_t + \frac{s}{r+t+s}\po\frac1s\int_{t}^{t+s} \delta_{X_u} \dd u\pf$. Together with \eqref{EqBorneQ}, that means $t\mapsto F_t(z)$ is almost everywhere differentiable with 
\begin{eqnarray*}
\left| \partial_t F_t(z)\right| & \leqslant & \frac{C_2 + \frac{C_1}{\rho}}{t^2}\| f\|_\infty.
\end{eqnarray*}
 Itô's formula reads
 \begin{eqnarray*}
\int_{s}^{t} L^{\mu_u} F_u(Z_u) \dd u 
  &=& F_t(Z_t)-F_s(Z_s) - \po M_t - M_s + \int_s^t \po \partial_u F_u\pf (Z_u)\dd u\pf,
 \end{eqnarray*}
 where $M_t-M_s$ is a martingale with  quadratic variation 
 \begin{eqnarray*}
\int_s^t \Gamma^{\mu_u}\po F_u\pf(Z_u)\dd u & \leqslant &  \lambda_{\max}\po \frac{C_1}{\rho}\| f\|_\infty\pf^2\po \frac1s-\frac1t\pf.
 \end{eqnarray*}
 From Doob's inequality, it means that  for any $T,\delta,t>0$,
\begin{eqnarray*}
\mathbb P\po \left. \underset{0\leq s \leq T}\sup  \left|M_{e^{t+s}}-M_{e^t}\right| \geqslant \delta \right|\ \mathcal F_{e^t}\pf &\leqslant & \frac{ \lambda_{\max}\po \frac{C_1}{\rho}\| f\|_\infty\pf^2}{\delta^2 e^t} .
\end{eqnarray*}
 On the other hand, 
\begin{eqnarray*}
| F_{e^{t+s}}(Z_{e^{t+s}})-F_{e^t}(Z_{e^t}) | & \leq & \frac{2C_1}\rho e^{-t}\| f\|_\infty
\end{eqnarray*}
and
 \begin{eqnarray*}
\left\| \int_{e^{t+s}}^{e^t} \po \partial_u F_u\pf (Z_u)\dd u\right\|_\infty 
& \leq & \po C_2 + \frac{C_1}{\rho}\pf e^{-t} \| f\|_\infty.
\end{eqnarray*}
Altogether, if $t$ and $\delta$ are such that $ \po C_2 + 3\frac{C_1}{\rho}\pf e^{-t}  \| f\|_\infty\leqslant \frac{\delta}{2}$ then
\begin{eqnarray*}
\mathbb P\po\left.  \underset{0\leq s\leq T}\sup |\varepsilon_t(s) f|\geqslant \delta \ \right|\ \mathcal F_{e^t} \pf  & \leqslant & \mathbb P\po \underset{0\leq s \leq T}\sup \left.\left|M_{e^{t+s}}-M_{e^t}\right| \geqslant \frac{\delta}{2}\right|\ \mathcal F_{e^t}\pf\\
& \leqslant & \frac{C_3 e^{-t}}{\delta^2} \| f\|_\infty^2
\end{eqnarray*} 
for some $C_3$. On the other hand a similar bound obviously holds   when $ \po C_2 + 3\frac{C_1}{\rho}\pf e^{-t} \| f\|_\infty > \frac{\delta}{2}$ since a probability is always less than $1 < e^t < e^t\po \frac{2}{\delta}\po C_2 + 3\frac{C_1}{\rho}\pf e^{-t} \| f\|_\infty\pf^2$.
\end{proof}

\subsection{Convergence toward sinks}\label{SectionODE2}

In this subsection, Assumptions \ref{HypoUnifnu}, \ref{HypoContiNu} and \ref{HypoSymmetric} are enforced. The two following Lemmas do not require any specific new argument with respect to the work of Benaïm and Raimond:

\begin{lem}\label{LemEpsiV}
There exists a constant $C_4$ such that for all $T,t >0$ and $\delta\in(0,1)$,
\begin{eqnarray}\label{EqEpsiV}
\mathbb P\po\left.  \underset{0\leq s\leq T}\sup \| V_{\varepsilon_t(s)} \|_\infty \geqslant \delta \ \right|\ \mathcal F_{e^t} \pf  & \leqslant &  \frac{C_4 e^{-t}}{\delta^3}.
\end{eqnarray} 
\end{lem}
\begin{proof}
This is a corollary of Proposition \ref{PropEpsi}, as proved in \cite[Lemma 5.3]{BenaimRaimond}.
\end{proof}

\begin{prop}\label{PropCVsink}
Let $\nu$ be a sink of $\Psi$. Then there exist an open neighbourhood $\mathcal U$ of $V_\nu$ in $\po \mathcal C^0\po\mathcal M\pf,\|\cdot\|_\infty\pf$ and $T,\delta>0$ such that for all $t\geqslant T$,
\begin{eqnarray*}
\mathbb P\po \mu_s \underset{s\rightarrow\infty}\longrightarrow \nu \pf & \geqslant & \po 1 - \frac{C_4 e^{-t}}{\delta^3}\pf \mathbb P\po \exists\ s\geqslant t\ \text{s.t.}\ V_{\mu_s}\in\mathcal U\pf 
\end{eqnarray*}
\end{prop}

\begin{proof}
The proof is similar to the one of \cite[Lemma 5.4]{BenaimRaimond}, namely from \eqref{EqEpsiV} is obtained a similar estimate (see \cite[Lemma 5.2]{BenaimRaimond}) which, together with \cite[Theorem 7.3]{Benaim99}, yields Proposition \ref{PropCVsink}.
\end{proof}

From Proposition \ref{PropCVsink} to Theorem \ref{ThmCVsink}, only a control argument is missing according to which the probability to be in the neighbourhood $\mathcal U$ is positive for large times. This is done in the following, which concludes the proof of Theorem \ref{ThmCVsink}:

\begin{prop}\label{PropControle}
Let $\nu \in \mathcal P\po \mathcal M\pf$ and $\mathcal U$ be an open neighbourhood of $V_\nu$ in $\mathcal C^0\po \mathcal M\pf$. Then there exists $T>0$ such that for all $t\geqslant T$,
\begin{eqnarray*}
\mathbb P\po V_{\mu_t} \in \mathcal U\pf  &>&0.
\end{eqnarray*}
\end{prop}

\begin{proof}
For a continuous function $\omega:\R_+\rightarrow \mathcal M$, we denote 
 \[V_{\omega,t,r}(x) =    \frac1{r+t}\int_0^t W(x,\omega(s))\dd s + \frac{r}{r+t} \int  W(x,u) \mu_0(\dd u).\]
For $t,c>0$, we denote by $\mathcal G_{t,c}$ the set of continuous trajectory on $\mathcal M$ of length $t$ which are piecewise geodesic (with a finite number of jumps) with constant scalar speed equal $c$.
 
  We claim that it is sufficient  to construct for any $t$ large enough and  any $x_0 \in \mathcal M$ a trajectory $\omega\in\mathcal G_{t,c}$, where $c>0$ is in the support of the probability $p$ given in Assumption~\ref{HypoUnifnu},  which starts at $x_0$ and such that $V_{\omega,t,r}\in\mathcal U$. Indeed, if we have such a trajectory, 
  for any $\varepsilon>0$ there is a positive probability that a SIVJP with initial position $x_0$ stays at distance less than $\varepsilon$ from the deterministic $ \omega$ up to time $t$. For $\varepsilon$ small enough, it implies $V_{\mu_{t}}\in\mathcal U$. So let us construct such an $\omega$, for fixed $c>0$ and $x_0\in\mathcal M$.

Let $\varepsilon>0$, and let $N\in \mathbb N$ and $(x_i)_{i\in\cco 1,N\ccf} \in \mathcal M^N$ be  such that the balls $B_i(\varepsilon)$ of center $x_i$ and radius $\varepsilon$ cover $\mathcal M$ and such that there exist weights $(p_i)_{i\in\cco 1,N\ccf}$ with $\sum p_i = 1$ such that
\[\underset{x\in\mathcal N}\sup \left|\int V(x,u) \nu(\dd u) - \sum_{i=1}^N V(x,x_i) p_i\right|\ \leqslant \ \varepsilon.\]
 For each $i$ and for any arbitrary time $t_i>0$, there exists small loops $\omega_i \in \mathcal G_{t_i,c}$ with $\omega_i(0) = \omega_i(t_i)=x_i$ and such that $\omega(s) \in B_i(\varepsilon')$ for all $s\in[0,t_i]$, where $\varepsilon'$ is chosen small enough so that $B_i(\varepsilon') \cap B_j(\varepsilon') = \emptyset$ for $i\neq j$. Similarly, there exist $t_0>0$ and $\omega_0 \in \mathcal G_{t_0,c}$ with $\omega_0(0) = \omega_0(t_0) = x_0$ and such that for all $i\in \cco 1,N\ccf$ there exists at least one time $s_i \in [0,t_0]$ with $\omega_0(s_i) = x_i$. Loops can be added to $\omega_0$ in the following way: set $\tilde \omega_0(s) = \omega_0(s)$ up to time $s_i$, set $\tilde \omega_0(s_i + s) = \omega_i(s)$ for $s\in[0,t_i]$, and set $\tilde \omega_0(s_i+t_i+s) = \omega_0(s_i+s)$ for $s\in[0,t-s_i]$. That way, for any  by adding to $\omega_0$ an arbitrary number of loops of arbitrary lengths, we can construct $\omega \in \mathcal G_{t,c}$ for some $t>0$ such that 
\begin{eqnarray*}
\underset{i\in\cco 1,N\ccf}\sup \left|\frac1t\int_0^t \mathbb 1_{\omega(s)\in B_i(\varepsilon')}\dd s - p_i\right| & \leqslant & \varepsilon.
\end{eqnarray*}
Note that, for all $x\in \mathcal M$,
\begin{eqnarray*}
\left|\frac1t \int_0^t W(x,\omega(s)) \dd s - \sum_{i=1}^N \po \frac1t\int_0^t \mathbb 1_{\omega(s)\in B_i(\varepsilon')}\dd s \pf W(x,x_i)\right| & \leqslant & \| \partial_u W\|_\infty \varepsilon'.
\end{eqnarray*}

All this shows that there exist $t_0>0$ and $\omega \in \mathcal G_{t_0,c}$ which starts and ends at $x_0$ and such that $V_{\omega,t_0,0} \in \mathcal U$. By periodicity, $\omega$ may be defined on $\R_+$. 
 Now for a fixed initial weight $r$, for any $k\in\mathbb N$ and $t\in [k t_0,(k+1)t_0)$, 
\[V_{\omega,t,r} \ = \ \frac{k t_0}{r+t} V_{\omega,t_0,0} + \frac{t - kt_0}{r+t} V_{\omega,t-kt_0,0} + \frac{r}{r+t} \int  W(x,u) \mu_0(\dd u).\]
 Hence,
 \[\left\| V_{\omega,t,r} -V_{\omega,t_0,0} \right\|_\infty \ \leqslant \ 2\frac{r+t_0}{r+kt_0} \left\| W\right\|_\infty \]
 and for $k$ larger than some $k_0$, $V_{\omega,t,r}\in \mathcal U$. We have proved that for any $t\geqslant T:=k_0 t_0$, a deterministic trajectory $\omega$ is such that $V_{\omega,t,r}\in \mathcal U$, which concludes.
\end{proof}

\subsection{Non-convergence toward saddles}\label{SectionODE3}

In this subsection, Assumptions \ref{HypoUnifnu}, \ref{HypoContiNu}, \ref{HypoSymmetric} and \ref{HypoMercer} are enforced, 
and $\mu^*\in\mathcal P\po\mathcal M\pf$ is a saddle of the flow induced by the vector field $F(\nu)=\pi(\nu)-\nu$. We denote by $\mathbb M\po\mathcal M\pf$ the set of measures on $\mathcal M$,  and for $m=0$ or 1, $\mathbb M_m\po\mathcal M\pf=\{ \nu\in \mathbb M\po\mathcal M\pf, \ \nu \mathbb 1 = m\}$, where $\mathbb 1$ is the constant function with value 1. We consider $(Z,\mu)=(X,Y,\mu)$ a SIVJP.

Let us recall some facts whose details can be found in \cite[Section 6.1-6.3]{BenaimRaimond}. There exists $\mathcal H\subset \mathcal C^0\po \mathcal M\pf$ endowed with an Hilbert norm $\| \cdot \|_{\mathcal{H}} \geqslant c \| \cdot\|_\infty$ for some $c>0$ (so that the identity from $\mathcal H$ to $\mathcal C^0\po \mathcal M\pf$ is continuous) and so that, moreover, the following holds:
\begin{itemize}
\item For all $\nu\in \mathbb M\po\mathcal M\pf$, $V_\nu \in \mathcal H$ and there exists $C>0$ such that
\begin{eqnarray}\label{EqNormeH}
\| V_\nu\|_{\mathcal H} & \leqslant & C \|\nu\|_1
\end{eqnarray}
where $\|\nu\|_1 = \sup\{|\nu f|, \| f\|_\infty \leqslant 1\}$.
\item There exist an Hilbert basis $(e_i)_{i\geqslant0}$  of $\mathcal H$ and a sequence $s\in \{-1,1\}^{\mathbb N}$ such that 
\begin{eqnarray}\label{EqWBON}
W(x,u) &=& \sum_{i\geqslant 0} s_i e_i(x) e_i(u),
\end{eqnarray}
where the convergence of the sum is uniform with respect to $x,u\in\mathcal M$.
\end{itemize}
Denoting by $\mathcal H_m$ the closure in $\mathcal H$ of $\{ V_{\nu},\ \nu \in \mathbb M_m\po\mathcal M\pf \}$, then
\[F^V(h)\ :=\ \int W(\cdot,u)\frac{e^{-h(u)}}{\int e^{-h(r)}\dd r}\dd u - h,\]
from $\mathcal H_1$ to $\mathcal H_0$, which satisfies $F^V\po V_{\nu}\pf = V_{F(\nu)}$, induces a global smooth flow $\Psi^V$ on $\mathcal H_1$ such that for all $\nu \in \mathcal H_1$ and $t\in \R$,
\[V_{\Psi_t(\nu)} = \Psi_t^V\po V_{\nu}\pf.\]
Moreover $h^* = V_{\mu^*}$ is a saddle of $\Psi^V$. We want to prove that $\mathbb P\po V_{\mu_t} \overset{\mathcal H}{\underset{t\rightarrow\infty}\longrightarrow} h^*\pf = 0$, which will imply Theorem \ref{ThmPasCVsaddle}, but now we work in an Hilbert space rather than on $\mathcal P\po\mathcal M\pf$. In particular, we have an orthogonal decomposition
\begin{eqnarray*}
\mathcal H &=& \mathcal H^u \oplus \mathcal H^s
\end{eqnarray*}
with $\mathcal H^u\neq\{0\}$ and such that for all $v$ in the unstable space $\mathcal H^u$ and $t\in\R$,
\[\| \mathcal D \Psi_t^V(h^*) v\|_{\mathcal H} \geqslant C e^{\lambda |t|} \|   v\|_{\mathcal H} \]
for some $C,\lambda>0$, where $\mathcal D$ stands for the differential operator.  Denoting by
\begin{eqnarray*}
Stab(h^*) &=& \left\{ h\in\mathcal H_1,\ \underset{t\rightarrow\infty} \lim \Psi^V_t(h) = h^*\right\}
\end{eqnarray*}
the basin of attraction of $h^*$, we can construct a $\mathcal C^2$ function $\eta$ from a neighbourhood $\mathcal N$ of $h^*$ in $\mathcal H$ to $\R_+$ such that the following holds:
\begin{itemize}
\item For all $h\in \mathcal N \cap Stab(h^*)$, $\eta\po h\pf =0$.
\item For all $h\in \mathcal N $,
\begin{eqnarray}\label{EqDetaY}
\mathcal D\eta(h) F^V(h) & \geqslant & 0.
\end{eqnarray}
\item For all $\varepsilon>0$, there exist a neighbourhood $\mathcal N_\varepsilon \subset \mathcal N$ of $h^*$ and $C_4>0$ such that, denoting by $\mathcal D^2$ the Hessian operator, for all $h\in\mathcal N_\varepsilon$ and $u,v\in \mathcal H_0$,
\begin{eqnarray}
|\mathcal D^2_{u,v} \eta(h) - D^2_{u,v} \eta(h^*)| & \leqslant & \varepsilon \| u\|_{\mathcal H}\| v\|_{\mathcal H}\label{Eqeta2}\\
|  \mathcal D^2_{u,v} \eta(h)| & \leqslant & C_4 \| u\|_{\mathcal H}\| v\|_{\mathcal H}\label{Eqeta3}\\
|\mathcal D\eta(h) v| &\leqslant & C_4\| v\|_{\mathcal H} \sqrt{\eta(h)}\label{Eqeta4}\\
2 \eta(h) \mathcal D^2_{v,v} \eta(h) -\po \mathcal D \eta(h) v\pf^2 & \geqslant & - C_4 \| v\|_{\mathcal H}^2 \po \eta(h)\pf^{\frac32}.\label{Eqeta5}
\end{eqnarray}
In particular, $\mathcal D^2_{v,v} \eta(h^*) \geqslant 0 $ for all $v\in \mathcal H_0$.
\item \begin{eqnarray}
  \mathcal D^2_{v,v} \eta(h^*) = 0& \Leftrightarrow &v\in H^s
\end{eqnarray}
\end{itemize}
The function $\eta$ is in some sense the square of some distance to $Stab(h^*)$. Indeed, away from $Stab(h^*)$, it necessarily increases along the flow $\Psi^V$ (which is \eqref{EqDetaY}); from \eqref{Eqeta4}, $\sqrt \eta$ is Lipschitz; and \eqref{Eqeta5} implies that, at the saddle $h^*$, $\eta$ is stricly convex in the unstable directions.

Here ends the recalls from \cite{BenaimRaimond}. The strategy is now the following: we will show that for any $L>0$, if $V_{\mu_t}$ is in $\mathcal N$ at time $t$, then $s\mapsto \eta\po V_{\mu_s}\pf$ can reach the level $\frac Lt$ with positive probability, and that from that level, if $L$ is large enough, it has a positive probability not to converge to 0, which will be contradictory with $\mu_t\rightarrow \mu^*$. 

More precisely, rather than with $\mu_t$ and $V_{\mu_t}$, we will work with
\begin{eqnarray*}
\nu_t f &:=& \mu_t f+ \frac1{r+t}Q^{\mu_t} f(Z_t)\\
g_t  &:=&  V_{\nu_t}.
\end{eqnarray*}
Namely, denoting by $V_x = W(x,\cdot) $,
\[g_t(x) = \mu_t V_x + \frac1{r+t} Q^{\mu_t} V_x (Z_t).\]
 Then, $(g_t)_{t\geqslant0}$ is a $\mathcal H_1$-valued semimartingale.
 
Given $L>0$ (to be chosen later on), a time $t>0$ and a neighbourhood $\mathcal N$ of $h^*$ in $\mathcal H$, we consider the following stopping times:
\begin{eqnarray*}
S_t &=& \inf\left\{ s\geqslant t,\ \eta(g_s) \geqslant \frac{L^2}{t}\right\}\\
U_t^{\mathcal N} &=& \inf\left\{ s\geqslant t,\ \mu_s \notin \mathcal N\right\}.
\end{eqnarray*}
We begin with the following technical lemma:
\begin{lem}\label{Lemetag}
There exists $C_5>0$ such that for all $\varepsilon>0$, there exist a neighbourhood $\mathcal N_\varepsilon \subset \mathcal N$ of $h^*$ in $\mathcal H$ and a time $t_0$ such that for all $t\geqslant t_0$,
\begin{eqnarray*}
  \mathbb E\po\left.  \eta\po g_{S_t\wedge U_t^{\mathcal N_\varepsilon}}\pf\ \right|\ \mathcal F_t\pf - \eta\po g_t\pf  &\geqslant &  \frac{-\varepsilon + C_5 \mathbb P\po\left.  S_t\wedge U_t^{\mathcal N_\varepsilon} = \infty \right |\ \mathcal F_t\pf }{r+t}. 
\end{eqnarray*} 
\end{lem} 
 \begin{proof}
Between two jumps of the velocity $Y$, the evolution of $g_t$ is deterministic and, for almost every $t\geqslant0$, $\partial_t g_t=B_t$ where
\begin{eqnarray*}
 B_t(x) &=& \frac{1}{r+t}\po V_{\delta_{X_t}}(x) - V_{\mu_t}(x)\pf + \co \partial_t\po \frac1{r+t} Q^{\mu_t}\pf  + \frac1{r+t} D Q^{\mu_t}\cf V_x (Z_t)\\
&=& \frac{1}{r+t}\po V_{\pi\po\mu_t\pf}(x) - V_{\mu_t}(x)\pf +  \co \partial_t\po \frac1{r+t} Q^{\mu_t}\pf  - \frac{\lambda(Z_t,\mu_t)}{r+t} \po H^{\mu_t} - Id\pf Q^{\mu_t}\cf V_x (Z_t)\\
&=& \frac{F^V(g_t)(x)}{r+t} +\frac{N_t V_x}{(r+t)^2}- \frac1{r+t} \lambda(Z_t,\mu_t)\po H^{\mu_t} - Id\pf Q^{\mu_t}  V_x (Z_t),
\end{eqnarray*}
where we used $L^\mu Q^\mu = K^\mu$, and defined $N_t$ by
\begin{eqnarray*}
N_t f &=& (r+t)^2 \partial_t\po \frac1{r+t} Q^{\mu_t}\pf f(Z_t) + (r+t)\po F(\mu_t)-F(\nu_t)\pf f.
\end{eqnarray*}
 Note that from Lemmas \ref{LemPK} and \ref{LemQ}, $\| N_t \|_1\leq C_6$ for some $C_6>0$. Writing $  g_t^{(v)}(x) = \mu_t V_x + \frac1{r+t} Q^{\mu_t} V_x (X_t,v)$, we thus have (using the definition of a SIVJP)
\begin{eqnarray*}
\eta\po g_t\pf - \eta(g_s) &=& M_t - M_s + \int_s^t \mathcal D\eta\po g_u\pf B_u +\lambda\po Z_u,\mu_u\pf \int \po \eta\po g_u^{(v)}\pf - \eta(g_u)\pf h^{\mu_u}(Z_t,\dd v)\dd u
\end{eqnarray*}
where $M_t$ is a martingale. Set $g^* = V_{\mu^*}$ and, for $u\geqslant0$, 
\begin{eqnarray*}
R_u  &=& \int \Big(\eta\po g_u^{(v)}\pf - \eta(g_u) \\
& &  -\co  \mathcal D\eta\po g_u\pf\po g_u^{(v)}-g_u\pf +\frac12\mathcal  D^2\eta(g^*) \po g_u^{(v)}-g_u, g_u^{(v)}-g_u\pf \cf \Big) h^{\mu_u}(Z_t,\dd v).
\end{eqnarray*}
From  \eqref{EqNormeH}, for some $C>0$, $\|  g_t^{(v)} - g_t\|_{\mathcal H}\leqslant \frac{C}{\rho(r+t)}$ for all $v$ which, together with Inequalities \eqref{Eqeta2} and \eqref{Eqeta3},  imply that for all $\varepsilon>0$ there exist a neighbourhood $\mathcal N_\varepsilon$ of $\mu^*$ and a time $t_0$ such that if $g_t \in \mathcal N_\varepsilon$ and $t\geqslant t_0$ then $|R_t|\leqslant \frac{\varepsilon}{(r+t)^2}$. On the other hand we can also choose $\mathcal N_\varepsilon$ so that  for all $\mu \in\mathcal N_\varepsilon$, $\|\lambda(\cdot,\mu) - \lambda(\cdot,\mu^*)\|_\infty\leqslant \varepsilon$, and  (according to \eqref{Eqeta4}), $|\mathcal  D\eta(V_\mu) |\leqslant \varepsilon$. Finally, denoting by
\begin{eqnarray*}
  g_t^{(v),*}(x)-g_t^*(x) &=& \frac1{r+t} \po   Q^{\mu^*} V_x(X_t,v)  -Q^{\mu^*} V_x (Z_t)\pf ,
\end{eqnarray*}
 we can also  choose (using \eqref{EqNormeH} and Assumption \ref{HypoContiNu}) $\mathcal{N}_\varepsilon$ and $t_0$ such that  
 \begin{eqnarray*}
\left\| \int  \po g_t^{(v),*} -g_t^* \pf h^{\mu^*}(Z_t,\dd v) -  \int \po g_t^{(v)}-g_t\pf  h^{\mu_t}(Z_t,\dd v) \right\|_{\mathcal H} &\leqslant & \frac\varepsilon{r+t} .
\end{eqnarray*}
In other words, writing
\begin{eqnarray}\label{Eqeta}
\eta\po g_t\pf - \eta(g_s)  &=& M_t - M_s  + \int_s^t \co   \frac{\mathcal D\eta\po g_u\pf   F^V(g_u)}{r+u} + \widetilde R_u  \cf \dd u \notag\\
 & & +   \frac12 \int_s^t\lambda\po Z_u,\mu^*\pf \int \mathcal D^2\eta(g^*) \po g_u^{(v),*}-g_u^*,  g_u^{(v),*}-g_u^*\pf h^{\mu^*}(Z_t,\dd v)  \dd u  ,
 \end{eqnarray}
with a remainder
\begin{eqnarray*}
\widetilde R_u &:=& \frac{ \mathcal D\eta\po g_u\pf N_u V}{(r+u)^2}  +\lambda\po Z_u,\mu_u\pf R_u    \\
& &  +\ \frac12 \int   \lambda\po Z_u,\mu_u\pf \mathcal  D^2\eta(g^*) \po  g_u^{(v)}-g_u, g_u^{(v)}-g_u\pf h^{\mu_t}(Z_t,\dd v)\\
& &  - \  \frac12 \int  \lambda\po Z_u,\mu^*\pf  \mathcal D^2\eta(g^*) \po  g_u^{(v),*}-g_u^*,\bar g_u^{(v),*}-g_u^*\pf    h^{\mu^*}(Z_t,\dd v)  , 
\end{eqnarray*}
we can choose a neighbourhood $\mathcal N_\varepsilon$ and a time $t_0$ such that for all $t_0\leqslant t\leqslant s\leqslant   S_t\wedge U_t^{\mathcal N_\varepsilon}$, $|\widetilde R_s| \leqslant \frac{\varepsilon}{(r+s)^2}$ (we used \eqref{EqNormeH} and \eqref{Eqeta3} again). Taking the expectation in \eqref{Eqeta}, the martingale increment vanishes, and together with \eqref{EqDetaY} and the fact $\mathcal D^2_{v,v}\eta \geqslant 0$ for all $v\in\mathcal H_0$, we have obtained so far
\begin{eqnarray*}
  \mathbb E\po\left.  \eta\po g_{S_t\wedge U_t^{\mathcal N_\varepsilon}}\pf\ \right|\ \mathcal F_t\pf - \eta\po g_t\pf  &\geqslant & - \frac{\varepsilon}{r+t} +
 \frac12 \mathbb E\po \left.\mathbb 1_{S_t\wedge U_t^{\mathcal N_\varepsilon} = \infty}\int_t^{\infty} \frac{1}{(r+u)^2}\Phi(Z_u)   \dd u\right |\ \mathcal F_t\pf
\end{eqnarray*}
with (recall that $W(x,y) = \sum s_i e_i(x) e_i(y)$)
\begin{eqnarray*}
\Phi\po z\pf &:=&   \sum_{i,j}\mathcal D^2_{i,j}\eta(g^*) s_i s_j \Gamma^{\mu^*}\po Q^{\mu^*} e_i , Q^{\mu^*} e_j\pf(z).
\end{eqnarray*}
From
\[\Phi(Z_u)  \ =\ \mu_u \Phi + (r+u)\partial_u\po \mu_u \Phi\pf,\]
an integration by part yields
\begin{eqnarray*}
 \int_t^{\infty} \frac{1}{(r+u)^2}\Phi(Z_u)   \dd u &=& - \frac{\mu_t \Phi}{r+t} + 2 \int_t^\infty \frac{\mu_u \Phi}{\po r+u\pf^2}\dd u.
\end{eqnarray*}
We can choose the neighbourhood $\mathcal N_\varepsilon$ so that for all $\mu \in \mathcal N_\varepsilon$, $|\mu \Phi - \mu^*\Phi|\leqslant \varepsilon$. That way, on the event $\{S_t \wedge U_t^{\mathcal N_\varepsilon} = \infty\}$,
\begin{eqnarray*}
 \int_t^{\infty} \frac{1}{(r+u)^2}\Phi(Z_u)   \dd u &\geqslant & \frac{1}{r+t}\po \mu^*\Phi - 3\varepsilon\pf
\end{eqnarray*}
and thus
\begin{eqnarray*}
  \mathbb E\po  \eta\po g_{S_t\wedge U_t^{\mathcal N_\varepsilon}}\pf\ |\ \mathcal F_t\pf - \eta\po g_t\pf  &\geqslant &  \frac{-4\varepsilon + \frac12 \po \mu^*\Phi\pf \mathbb P\po\left.  S_t\wedge U_t^{\mathcal N_\varepsilon} = \infty \right |\ \mathcal F_t\pf }{r+t}. 
\end{eqnarray*}
It remains to prove that $\mu^*\Phi >0$. Since $\mu^*$ is invariant for $L^{\mu^*}$,
\begin{eqnarray*}
& &\int \Gamma^{\mu^*}\po Q^{\mu^*} e_i , Q^{\mu^*} e_j\pf  \mu^* \po \dd z\pf \\
&=& -\int \frac{\po L^{\mu^*}  Q^{\mu^*} e_i \pf Q^{\mu^*} e_j +  Q^{\mu^*} e_i \po L^{\mu^*}  Q^{\mu^*} e_j \pf}2  \mu^* \po \dd z\pf \\
&=& \int \int_0^\infty \frac{\po K^{\mu_*}e_i\pf\po P_t^{\mu^*} K^{\mu^*}  e_j\pf+\po P_t^{\mu^*} K^{\mu^*}  e_i\pf\po K^{\mu_*}e_j\pf}2   \dd t\   \mu^* \po \dd z\pf .
\end{eqnarray*}
Denoting by $\po P_t^{\mu^*}\pf'$ the adjoint of $P_t^{\mu^*}$ in $L^2\po\mu^*\pf$, we write
\begin{eqnarray*}
\int K^{\mu^*}  e_i \frac{\po P_t^{\mu^*}\pf'+ P_t^{\mu^*}}2 K^{\mu^*}  e_j\ \mu^* \po \dd z\pf 
&=&  \int \po R_t K^{\mu^*}  e_i \pf\po R_t K^{\mu^*}  e_j\pf \mu^* \po \dd z\pf  
\end{eqnarray*}
with $R_t$ a square root of the self-adjoint operator $\frac12\po \po P_t^{\mu^*}\pf'+ P_t^{\mu^*}\pf $, and more precisely the unique square root  which is a Markov operator. Writing
\begin{eqnarray*}
v^z_t &:=& \sum_i s_i \po R_t K^{\mu^*}  e_i(z) \pf e_i\\
&=& V_{\po \delta_z R_t\pf} - V_{\mu^*},
\end{eqnarray*}
we obtain, for each $t\geqslant 0$,
\begin{eqnarray*}
& & \int\sum_{i,j}\mathcal D^2_{i,j}\eta(g^*) s_i s_j  \po K^{\mu^*}  e_i\pf\po \frac{\po P_t^{\mu^*}\pf'+ P_t^{\mu^*}}2 K^{\mu^*}  e_j\pf \mu^* \po \dd z\pf \\
&=&   \int\sum_{i,j} \mathcal D^2_{i,j}\eta(g^*) s_i s_j  \po R_t K^{\mu^*}  e_i \pf\po R_t K^{\mu^*}  e_j\pf \mu^* \po \dd z\pf  \\
&=&   \int\mathcal  D^2 \eta(g^*) \po v^z_t, v^z_t\pf \mu^* \po \dd z\pf\\
&\geqslant & 0.
\end{eqnarray*} 
Hence, in order to prove
\[\mu^* \Phi = \int_0^\infty \int\sum_{i,j}D^2_{i,j}\eta(g^*) s_i s_j  \po K^{\mu^*}  e_i\pf\po \frac{\po P_t^{\mu^*}\pf'+ P_t^{\mu^*}}2 K^{\mu^*}  e_j\pf \mu^* \po \dd z\pf  \dd t > 0\]
 it is sufficient to prove that 
\begin{eqnarray*}  0 &<&  \int\sum_{i,j}\mathcal D^2_{i,j}\eta(g^*) s_i s_j  \po K^{\mu^*}  e_i\pf\po   K^{\mu^*}  e_j\pf \mu^* \po \dd z\pf  \\
& = &  \int\mathcal  D^2 \eta(g^*) \po v^z_0, v^z_0\pf \mu^* \po \dd z\pf. 
\end{eqnarray*}  
Since $\mu^* = \pi(\mu^*)$ admits a positive density with respect to the Lebesgue measure, $\mu^*\Phi =0$ would imply that $\mathcal D^2 \eta(g^*) \po v^z_0, v^z_0\pf = 0 $ for all $z \in \mathcal M$, or in other words that $v^z_0 \in\mathcal  H^s$ for all $z\in M$. This would imply that $V_{x_1}-V_{x_2}\in\mathcal  H^s$ for all $x_1,x_2\in \mathcal M$, and therefore, $V_m \in\mathcal  H^s$ for all $m\in \mathbb M_0\po\mathcal M\pf$. This would be in contradiction with the fact that $\mathcal H^u\neq\{0\}$. Thus, $\mu^*\Phi >0$.
\end{proof}
This first lemma yields the following:
\begin{lem}\label{LemStUt}
There   exist a neighbourhood $\mathcal N_\varepsilon \subset \mathcal N$ of $h^*$ in $\mathcal H$, a time $t_0$ and $p>0$  such that for all $t\geqslant t_0$,
\begin{eqnarray*}
  \mathbb P\po\left.S_t \wedge U_t^{\mathcal N_\varepsilon} <\infty\ \right| \ \mathcal F_t \pf & \geqslant & p. 
\end{eqnarray*} 
\end{lem}
\begin{proof}
From Lemma \ref{Lemetag} and
\begin{eqnarray*}
  \mathbb E\po\left.  \eta\po g_{S_t\wedge U_t^{\mathcal N_\varepsilon}}\pf\ \right|\ \mathcal F_t\pf - \eta\po g_t\pf  &\leqslant &  \mathbb E\po\left.  \frac{L^2}{S_t\wedge U_t^{\mathcal N_\varepsilon}}\ \right|\ \mathcal F_t\pf  ,
\end{eqnarray*} 
we get
\begin{eqnarray*}
  \mathbb P\po\left. S_t \wedge U_t^{\mathcal N_\varepsilon} < \infty\ \right| \ \mathcal F_t \pf & \geqslant & \mathbb E\po\left.  \frac{r+t}{S_t\wedge U_t^{\mathcal N_\varepsilon}}\ \right|\ \mathcal F_t\pf \\
   &\geqslant & \frac{-\varepsilon + C_5 \mathbb P\po\left.  S_t\wedge U_t^{\mathcal N_\varepsilon} = \infty \right |\ \mathcal F_t\pf}{L^2}.
\end{eqnarray*} 
Therefore,
\begin{eqnarray*}
  \mathbb P\po\left. S_t \wedge U_t^{\mathcal N_\varepsilon} < \infty\ \right| \ \mathcal F_t \pf    &\geqslant & \frac{-\varepsilon + C_5  }{L^2 + C_5},
\end{eqnarray*}
which concludes for $\varepsilon< C_5$.
\end{proof}
Let $\mathcal N_\varepsilon  $,  $t_0$ and $p $ be as in Lemma \ref{LemStUt}, and consider the event
\begin{eqnarray*}
\mathcal G &=& \left\{\liminf \ \eta\po g_t\pf \ >\ 0\right\}.
\end{eqnarray*}
\begin{lem}\label{LemH}
For $L$ large enough, there  exists $t_1>0$ such that for all $t\geqslant t_1$, on the event $\{S_t < U_t^{\mathcal N_\varepsilon} = \infty\}$,
\begin{eqnarray*}
\mathbb P\po \mathcal G\ |\ \mathcal F_{S_t}\pf & \geqslant & \frac12.
\end{eqnarray*}
\end{lem}
\begin{proof}
Making use of the notations introduced in the proof of Lemma \ref{Lemetag},
\begin{eqnarray*}
\sqrt{\eta\po g_t\pf} - \sqrt{\eta(g_s)} &=& A_t - A_s + \int_s^t \Big [\frac{\mathcal D\eta\po g_u\pf B_u}{2\sqrt{\eta\po g_u\pf}}\\
& & +\ \lambda\po Z_u,\mu_u\pf\int \po \sqrt{\eta\po  g_u^{(v)}\pf} - \sqrt{\eta\po g_u\pf}\pf h^{\mu_u}(Z_u,\dd v)\Big]\dd u
\end{eqnarray*}
where $A_t$ is an $\mathcal F_t$-martingale. Let
\begin{eqnarray*}
I_t &=& \underset{s\in\co S_t,U_t^{\mathcal N_\varepsilon}\cf}\inf \po A_s - A_{S_t}\pf
\end{eqnarray*}
and $T_t = \inf\{s>S_t,\ \eta(g_s)=0\}$. On the event $\{S_t<U_t^{\mathcal N_\varepsilon}\}$, for $s\in[S_t,T_t\wedge U_t^{\mathcal N_\varepsilon})$,
\begin{eqnarray*}
\sqrt{\eta\po g_s\pf}  & \geqslant & \frac{L}{\sqrt{S_t}} + I_t + \int_{S_t}^s \po \frac{\mathcal D\eta\po g_u\pf  }{2\sqrt{\eta\po g_u\pf}}\po  \frac{F^VY(g_u)}{r+u} +\frac{N_u V}{(r+u)^2}\pf +\lambda\po Z_u,\mu_u\pf\widetilde R_u \pf \dd u\\
& &  +\frac12 \int_{S_t}^s \lambda\po Z_u,\mu_u\pf \int \mathcal D^2 \sqrt \eta\po g_u\pf\po  g_u^{(v)} - g_u,  g_u^{(v)} - g_u\pf  h^{\mu_u}(Z_u,\dd v)\dd u
\end{eqnarray*}
with 
\begin{eqnarray*}
\widetilde R_u &=& \int \Big[ \sqrt{\eta\po  g_u^{(v)}\pf} - \sqrt{\eta\po g_u\pf} - \mathcal D\sqrt{\eta}\po g_u\pf\po g_u^{(v)} - g_u\pf \\
& & -\ \frac12 \mathcal D^2 \sqrt \eta\po g_u\pf\po   g_u^{(v)} - g_u,  g_u^{(v)} - g_u\pf \Big] h^{\mu_u}(Z_u,\dd v),
\end{eqnarray*}
which satisfies $|\widetilde R_u| \leqslant \frac{C_7}{(r+u)^2}$ for some $C_7$. Hence, on the event $\{S_t<U_t^{\mathcal N_\varepsilon}\}\cap \left\{I_t \geqslant - \frac{L}{2\sqrt{S_t}}\right\}$, for $s\in[S_t,T_t\wedge U_t^{\mathcal N_\varepsilon}]$, for some $C_8$,
\begin{eqnarray*}
\sqrt{\eta\po g_s\pf}  & \geqslant & \frac{L}{2\sqrt{S_t}} - \frac{C_8}{ S_t}.    
\end{eqnarray*}
For $t\geqslant t_1$ large enough, this is greater than $\frac{L}{4\sqrt{S_t}}$. For such $t\geqslant t_1$, thus, 
\begin{eqnarray*}
\{S_t<U_t^{\mathcal N_\varepsilon} = \infty\}\cap \left\{I_t \geqslant - \frac{L}{2\sqrt{S_t}}\right\} &\subset & \mathcal G.
\end{eqnarray*}
On the other hand, by Doob inequality, on the event $\{S_t<U_t^{\mathcal N_\varepsilon}= \infty\}$, for some $C_9,C_{10}$,
\begin{eqnarray*}
& & \mathbb P\po I_t < - \frac{L}{2\sqrt{S_t}}\pf \\
& \leqslant & \frac{4 S_t}{L^2} \mathbb E\po\left.\int_{S_t}^s  \lambda\po Z_u,\mu_u\pf \int \po \sqrt \eta\po   g_u^{(v)}\pf- \sqrt \eta\po g_u\pf \pf^2 h^{\mu_u}(Z_u,\dd v) \dd u \ \right|\ \mathcal F_{S_t}\pf \\
 & \leqslant & \frac{4 S_t C_9}{L^2} \mathbb E\po\left.\int_{S_t}^s  \int \|   g_u^{(v)} -   g_u  \|_{\mathcal H}^2 h^{\mu_u}(Z_u,\dd v)\dd u\ \right|\ \mathcal F_{S_t}\pf\\
 & \leqslant & \frac{4 S_t C_{10}}{L^2\po r+S_t\pf} . 
\end{eqnarray*}
We conclude by choosing  $L^2>8C_{10}$.
\end{proof}
 
\begin{proof}[Proof of Theorem \ref{ThmPasCVsaddle}]
We follow \cite[Proof of Theorem 2.26]{BenaimRaimond} (with a slight modification: we believe there was something unclear in the latter about the event $\{U_t^{\mathcal N_\varepsilon}<\infty\}$). We fix $L$, $t_0$, $t_1$ and $\mathcal N_\varepsilon$ as in Lemmas \ref{LemStUt} and \ref{LemH}, and let $A=\{\exists t>0,U_t^{\mathcal N_\varepsilon} = \infty\}$. For $t\geqslant t_0 \vee t_1$,
\begin{eqnarray}
\mathbb P\po \mathcal G\ |\ \mathcal F_t \pf & \geqslant & \mathbb E\po \left.\mathbb 1_{\mathcal G} \mathbb 1_{S_t < U_t^{\mathcal N_\varepsilon} = \infty}\ \right|\ \mathcal F_t \pf \notag\\
& \geqslant & \frac12 \mathbb P\po \left. S_t < U_t^{\mathcal N_\varepsilon} = \infty \ \right|\ \mathcal F_t \pf \notag\\
& \geqslant & \frac12 \po p -\mathbb P\po \left.  U_t^{\mathcal N_\varepsilon} < \infty \ \right|\ \mathcal F_t \pf\pf.\label{EqPH}
\end{eqnarray}
Almost surely,
\begin{eqnarray*}
\mathbb P\po \mathcal G\ |\ \mathcal F_t \pf & \underset{t\rightarrow \infty}\longrightarrow & \mathbb 1_{\mathcal G}\\
  \mathbb 1_{U_t^{\mathcal N_\varepsilon} = \infty} & \underset{t\rightarrow \infty}\longrightarrow & \mathbb 1_{A},
\end{eqnarray*}
so that
\begin{eqnarray*}
\mathbb E\po \left|\mathbb P\po U_t^{\mathcal N_\varepsilon} = \infty\ |\ \mathcal F_t\pf - 1_{A}\right|\pf &\leqslant & \mathbb E\po \left|\mathbb P\po A\ |\ \mathcal F_t\pf - 1_{A}\right|\pf  + \mathbb E\po \left|1_{U_t^{\mathcal N_\varepsilon} = \infty} - 1_{A}\right|\pf\\
& \underset{t\rightarrow \infty}\longrightarrow & 0.
\end{eqnarray*}
In other words, $\mathbb P\po U_t^{\mathcal N_\varepsilon} < \infty\ |\ \mathcal F_t\pf$ converges in the $L^1$ sense toward $1_{A^c}$, and letting $t$ go to infinity in \eqref{EqPH} yields
\begin{eqnarray*}
1_{\mathcal G} & \geqslant & \frac12 \po p - \mathbb 1_{A^c}\pf
\end{eqnarray*}
almost surely, which implies $A\subset \mathcal G$. Finally, almost surely
\[\{\mu_t\rightarrow \mu^* \} \subset A \subset \mathcal G \subset \{\mu_t\nrightarrow \mu^* \} \]
so that $\mathbb P\po  \mu_t\rightarrow \mu^*\pf =0$.
\end{proof}

\section{Quadratic interaction} \label{SectionQuadra}

In this section we consider the settings of Theorem \ref{ThmWQuadra}, and in particular the self-interacting potential on $\mathbb T$
\begin{eqnarray}\label{EqDefiWquadra}
W(x,z) &=& \rho \po \frac12|e^{i x} - e^{iz}|^2 - 1\pf = -\rho \cos(x-z)
\end{eqnarray}
for some $\rho\in \R$.  If $\rho>0$, $W$ is a self-attraction potential, if $\rho<0$ it is a self-repulsion one. We want to understand how the long-time behaviour of a SITP with such a self-interaction potential is affected by $\rho$ and by the external potential $U$.

\bigskip

First, let us remark that Assumptions \ref{HypoUnifnu}, \ref{HypoContiNu}, \ref{HypoSymmetric} and \ref{HypoMercer} are  satisfied in this case. As remarked in the Introduction, Assumptions \ref{HypoSymmetric} and \ref{HypoMercer}  are a consequence of \cite[Section 1.2]{MonmarcheRTP}. To check Assumptions  \ref{HypoUnifnu} and \ref{HypoContiNu}, set $\lambda^{\nu}(x,y) = 2\lambda_{\min} + \po y\partial_x V_\nu(x)\pf_+$ and
\begin{eqnarray*}
H^{\nu} f(x,y) & =& \frac{\lambda_{\min} + \po y\partial_x V_\nu(x)\pf_+}{\lambda^{\nu}(x,y)} f(x,-y) + \frac{\lambda_{\min}}{\lambda^{\nu}(x,y)} f(x,y). 
\end{eqnarray*}
Then $L^{\nu} = D + \lambda^\nu\po H^\nu - I\pf$, and for a positive $f$,
\begin{eqnarray*}
Q^{\nu} f(x,y) & \geqslant & \frac{\lambda_{\min}}{2 \lambda_{\min} + \| \partial_x W\|_\infty} \po f(x,-y) + f(x,y)\pf.
\end{eqnarray*}
All the other conditions are clear. Hence, Theorems \ref{ThmFix}, \ref{ThmCVsink} and \ref{ThmPasCVsaddle} hold.

\bigskip

Recall the notation
\begin{eqnarray*}
\overline{\pi}_\rho(a,b)(\dd z) & = & \frac{e^{-U(z) + \rho\po a\cos(z)+b\sin(z)\pf}}{\int e^{-U(x) + \rho\po a\cos(x)+b\sin(x)\pf} \dd x}\dd z.
\end{eqnarray*}
When there is no ambiguity on the value of $\rho$, we simply write $\overline{\pi}(a,b)$.

\subsection{Without exterior potential}

Bena\"im, Ledoux and Raimond studied in \cite{BenaimLedouxRaimond} the self-interacting diffusion on the sphere with a quadratic self-interaction and no exterior potential. When $U$ is a constant function, we recover in the piecewise deterministic case the same behaviour as in the diffusion case, which is the following: if the force is self-repulsing or if the self-attraction is not too strong, then the process does not localize, in the sense its empirical measure converges to the uniform measure on the circle, so that the particle behaves at infinity like an integrated telegraph process on the circle (as studied in \cite{Volte-Face}). In contrast, when the self-attraction is strong enough, a random direction is picked and the empirical measure of the process goes to a Gaussian law centered at this direction.

More precisely, 
it is proven in \cite[Lemma 4.8]{BenaimLedouxRaimond} that the equation
\begin{eqnarray*}
\int \cos  \dd \overline \pi_\rho (r,0) & = & r
\end{eqnarray*}
admits a positive solution, denoted by $r(\rho)$, if and only if $\rho>2$ (as far as notations are concerned, $(a,\beta)$ used in \cite{BenaimLedouxRaimond} correspond to $(\rho,r)$ used here by $\rho=4a$ and $\beta = r\rho$). 

\begin{proof}[Proof of Theorem \ref{ThmWQuadra}, point 1]
The arguments are exactly those of \cite[proof of Theorem 4.5]{BenaimLedouxRaimond}, which rely only on the limiting deterministic flow, which is the same in the diffusion and the PDMP case, on the fact that the empirical measure of the self-interaction process is an asymptotic pseudo-trajectory of this flow, and on the estimate \eqref{EqPropEpsi}. 
\end{proof}

\subsection{Stability of the Gibbs measure}

Let $U:\mathbb T\rightarrow \R$ be a smooth potential and denote by $m_U = \overline\pi(0,0)$ the associated Gibbs measure. If $(Z,\mu)$ is an SITP with potential $U$ and with a constant $W$ (namely there is no self-interaction and $Z$ is a Markov process) then $\mu_t \rightarrow m_U$. When $W $ is not constant, a necessary condition for $\mu_t \rightarrow m_U$ is that $\pi \po m_U\pf = m_U$, or in other words, $x\mapsto \int W(x,\cdot)\dd m_U$ is constant. When $W$ is given by \eqref{EqDefiWquadra}, this reads
\begin{eqnarray}\label{EqquadramU0}
& & \int \cos \dd m_U\ =\ \int \sin\dd m_U\ =\ 0.
\end{eqnarray}

\begin{prop}\label{PropStabGibbs}
Let $(Z,\mu)$ be a SITP with potential $U$ and $W$, where $W$ is given by \eqref{EqDefiWquadra} and $U$ is such that \eqref{EqquadramU0} holds. Let
\[\eta = \po \int \sin^2 \dd m_U\pf\po \int \cos^2 \dd m_U\pf - \po \int \sin\times \cos \dd m_U\pf^2 \hspace{10pt}\in \po 0, \frac14\cf.\]
 Then 
\begin{eqnarray*}
\rho \po\frac12 + \sqrt{\frac14 -\eta }\pf <1\hspace{20pt} & \Rightarrow & \hspace{20pt} \mathbb P\po \mu_t \rightarrow m_U\pf \ >\ 0 \\
\rho \po\frac12 + \sqrt{\frac14 -\eta }\pf >1\hspace{20pt} & \Rightarrow & \hspace{20pt} \mathbb P\po \mu_t \rightarrow m_U\pf \ =\ 0
\end{eqnarray*} 
\end{prop}
\textbf{Remark:} In particular if $\rho \leqslant 0$ (self-repulsion), $\mathbb P\po \mu_t \rightarrow m_U\pf \ >\ 0 $.
\begin{proof}
Note that $\pi(\nu) = \overline \pi (\bar \nu )$ with
\begin{eqnarray*}
\bar \nu &=& \po \int \cos \dd \nu ,\int \sin \dd \nu \pf .
\end{eqnarray*}
Let $\Psi$ be the flow on $\mathcal P\po \mathbb T\pf$ induced by $F(\nu) = \pi(\nu) - \nu$, and $\nu_t = \Psi_t(\nu)$. Then  $(a_t,b_t) = \bar \nu_t$ solves $\partial_t(a_t,b_t) = \overline F(a_t,b_t)$ where
\[\overline F(a,b) = \begin{pmatrix}
 \int \cos \dd \overline\pi (a,b)  -a\\
  \int \sin \dd \overline\pi (a,b)  -b
\end{pmatrix}.\]
By assumption, $\overline F(0,0)=0$ and we compute $J_{\overline F}(0,0)$ the Jacobian of $\overline F$ at $(0,0)$:
\begin{eqnarray*}
J_{\overline F}(0,0) &=&  \begin{pmatrix}
\rho\int \cos^2 \dd m_U -1 & \rho \int \cos \times \sin \dd m_U\\
\rho\int \cos \times \sin \dd m_U & \rho\int \sin^2 \dd m_U -1
\end{pmatrix}\hspace{10pt} := \rho M_U - I.
\end{eqnarray*}
The matrix $M_U$ is symmetric definite positive since $TrM_U = 1$ and $\det M_U = \eta>0$ (the strict positivity of $\eta$ comes from the fact $m_U$ admits a positive density with respect to the Lebesgue measure). The eigenvalues of $J_{\overline F}(0,0)$ are $\rho \po\frac12 \pm \sqrt{\frac14 -\eta }\pf -1$.

First, suppose $\rho \po\frac12 + \sqrt{\frac14 -\eta }\pf >1$. It implies $(0,0)$ is a saddle for the flow induced by $\overline F$ on $\R^2$, so that $m_U$ is a saddle for the flow induced by $F$ on $\mathcal P\po \mathbb T\pf$, and  Theorem \ref{ThmPasCVsaddle} concludes this case.

Second, suppose $\rho \po\frac12 + \sqrt{\frac14 -\eta }\pf <1$. Then $(0,0)$ is a sink for the flow induced by $\overline F$. Let $G:\mathcal P\po\mathbb T\pf\rightarrow \R^2$ be the mapping defind by $G(\nu)=\bar\nu$, and let
\[A\ = \ G^{-1}(0,0) \ =\ \left\{ \nu\in\mathcal P\po\mathbb T\pf,\ \int \cos \dd \nu = \int \sin \dd \nu = 0\right\}.\]
For $\nu\in A$, $\Psi_t\po\nu\pf = e^{-t}\po \nu - m_U\pf + m_U$, so that $m_U$ is a global attractor for the restriction $\Psi_{|A}$. Hence, $m_U$ is a sink for $\Psi$, and Theorem \ref{ThmCVsink} concludes.
\end{proof}

\subsection{Phase transition in a symmetric double-well potential}

The equation $\pi(\nu)=\nu$ with a double-well exterior potential $U$ (on $\R$ rather than $\mathbb T$), together with a quadratic attraction potential $W$, has been studied by Tugaut in  \cite{Tugaut2014}, motivated by the study of the McKean-Vlasov equation \eqref{EqMCKean}.  His ideas may be adapted  to our context.
\begin{prop}\label{PropDeuxpuits}
Suppose $U$ satisfies \eqref{EqquadramU0} and $U(z)=U(\pi - z)$ for all $z$. Write $A_0\ = \ \left\{\nu\in \mathcal P\po\mathbb T\pf,\ \int \sin \dd \nu = 0\right\}$.
\begin{itemize}
\item If $\rho\int \cos^2 \dd m_U \leqslant 1$ then 
\[A_0 \cap Fix(\pi) \ = \ \{ m_U\}.\]
\item If $\rho \int \cos^2 \dd m_U  >1$ then $A_0 \cap Fix(\pi)$ contains exactly three points, $m_U$ is a saddle of the flow $\Psi$ and the two other points of $A_0 \cap Fix(\pi)$ are sinks for the restriction of $\Psi$ to $A_0$. 
\end{itemize}
\end{prop}
\begin{proof}
We keep the notations of the previous section. The symmetry assumption on $U$ implies, in particular, that $\eta = \po \int \cos^2 \dd m_U\pf\po 1-\int \cos^2 \dd m_U\pf$ and that the eigenvalues of $J_F(0,0)$ are $\rho\int \cos^2 \dd m_U- 1$ and $\rho\int \sin^2 \dd m_U - 1$.

The points of $Fix(\pi)$ are all of the form $\overline \pi(a,b)$ where $(a,b)$ solves
 \[\left\{\begin{array}{rcl}
 a &=& \frac{\int \cos(z) \exp\po{-  U(z)+\rho\po a\cos(z)+b\sin(z)\pf}\pf\dd z}{\int  \exp\po{-  U(z)+\rho\po a\cos(z)+b\sin(z)\pf}\pf\dd z}\\
 & & \\
 b &=& \frac{\int \sin(z) \exp\po{-  U(z)+\rho\po a\cos(z)+b\sin(z)\pf}\pf\dd z}{\int  \exp\po{-  U(z)+\rho\po a\cos(z)+b\sin(z)\pf}\pf\dd z}.
 \end{array}\right.\]
 Multiplying both sides of these equations by $\int  e^{- U(z)+\rho\po a\cos(z)+b\sin(z)\pf}\dd z$, this is equivalent to
  \[\left\{\begin{array}{rcl} 0&=&\int (\cos(z)-a) e^{-  U(z)+\rho\po a\cos(z)+b\sin(z)\pf}\dd z \\
   & & \\
   0 &=& \int (\sin(z)-b) e^{-  U(z)+\rho\po a\cos(z)-b\sin(z)\pf}\dd z. \end{array}\right.\tag{*}\]
By assumption, $b=0$ always solves the second part. We are led to study the zeros of the function
\begin{eqnarray*}
\xi (a) & = & \int (\cos(z)-a) e^{- U(z)+\rho  a\cos(z) }\dd z.
\end{eqnarray*}
Expanding $a\mapsto   e^{ \rho  a\cos(z) }$ and using the symmetry of $U$  yields
\begin{eqnarray*}
\xi (a) & = & \sum_{n\geq 0}\frac{a^{2n+1}\rho^{2n}}{(2n)!} I (2n)\po \frac{\rho I (2n+2)}{(2n+1)I (2n)}-1\pf,
\end{eqnarray*}
 where $I (k) = \int \cos^k(z) e^{-  U(z)}\dd z$. Obviously, if $\rho \leqslant 0$, $a\xi(a) < 0$ as soon as $a\neq 0$, while $\xi(0)=0$, which concludes the proof in this case. In the following, we suppose $\rho>0$. Since $n\mapsto I (2n)$ is decreasing, $\frac{\rho I (2n+2)}{(2n+1)I (2n)}-1$ is decreasing and non-positive for $n$ large enough. Let
 \begin{eqnarray*}
 n_c  & = & \min\left\{n,\ \frac{I (2n+2)}{(2n+1)I (2n)} \  \leqslant\ \frac1\rho\right\}
 \end{eqnarray*}
 so that 
 \begin{eqnarray*}
\xi (a) & = & \sum_{n=0}^{ n_c -1}\frac{|\xi ^{(2n+1)}(0)|}{(2n+1)!} a^{2n+1} -\sum_{n\geq n_c } \frac{|\xi ^{(2n+1)}(0)|}{(2n+1)!} a^{2n+1}\\
& = & a^{2n_c+1}\po \sum_{n=0}^{ n_c -1}\frac{|\xi ^{(2n+1)}(0)|}{(2n+1)!} a^{2(n-n_c)} -\sum_{n\geq n_c } \frac{|\xi ^{(2n+1)}(0)|}{(2n+1)!} a^{2(n-n_c)}\pf.
\end{eqnarray*}
Both terms of the sum are non-increasing with $a> 0$, which implies $\xi $ admits at most one positive zero. Since it is an odd function, it admits either one zero (at 0) or three (at $-a_*,0$ and $a_*$ for some unique $a_*>0$). Note that $\xi(a_*) = 0$ implies $a_* = \int \cos\ \dd \overline \pi(a_*,0) < 1$. Differentiating $\xi $ with respect to $a$, we see that 
 \begin{eqnarray*}
\xi '(a) & = & a^{2n_c}\po \sum_{n=0}^{ n_c -1}\frac{|\xi ^{(2n+1)}(0)|}{(2n)!} a^{2(n-n_c)} -\sum_{n\geq n_c } \frac{|\xi ^{(2n+1)}(0)|}{(2n)!} a^{2(n-n_c)}\pf
\end{eqnarray*}
is even, vanishes at most once on $(0,\infty)$ and goes to $-\infty$ at infinity. There are two possibilities:
\begin{itemize}
\item if $\xi '(0) \leq 0$ then $\xi '(a)\leq 0$ for all $a\in \R$ and 0 is the only zero of $\xi$.
\item if $\xi '(0) > 0$ then $\xi $ is positive close to zero and goes to $-\infty$ at infinity, so that it vanishes three times.
\end{itemize}
Finally $\xi '(0)$ has the same sign as
\begin{eqnarray*}
\frac{\rho I (2)}{I (0)}-1 & = & \rho\int \cos^2 \dd m_U- 1.
\end{eqnarray*}   
In the case where $\xi'(0) >0$, we have $\xi'(a_*) = \xi'(-a_*)<0$. Since
\[\partial_a \po \frac{\xi(a)}{\int e^{-U(z)+\rho a\cos(z)}\dd z}\pf = \frac{\xi'(a)}{\int e^{-U(z)+\rho a\cos(z)}\dd z},\]
it means $\pm a_*$ are asymptotically stable equilibria for the flow induced by $\overline F$. The conclusion is now similar to the one of Proposition \ref{PropStabGibbs}, namely we remark $\overline\pi\po a_*,0\pf$ is a global attractor for the restriction of $\Psi$ to $G^{-1}( a_*,0)$, and similarly for $-a_*$.
\end{proof}

Obviously, if $U(z) = U(-z)$, denoting by $A_0' \ = \ \left\{\nu\in \mathcal P\po\mathbb T\pf,\ \int \cos \dd \nu = 0\right\} $, a similar argument (or a change of variables) proves that  $A_0'\cap Fix(\pi)$ is either reduced to $\{ m_U\}$ or constituted of three points, depending on the position of $\rho\int \sin^2 \dd m_U$ with respect to 1. The potential $U(z) = - cos(2z)$ being both symmetric with respect to the horizontal and vertical axes, we already know that there may be one, three or five fixed points on the axes. The rest of the disk remains to be studied.

The proof of the following has been kindly indicated to us by Jean-Baptiste Bardet, Michel Benaïm, Florent Malrieu and Pierre-André Zitt, and will appear in a work of them (for now in progress) about self-interacting processes.

\begin{lem}
Suppose $U(z) = -\cos(2z)$. Then
\begin{eqnarray*}
(A_0 \cup A_0')^c \cap Fix(\pi) & =& \emptyset.
\end{eqnarray*}
\end{lem}
\begin{proof}
We keep the previous notation $(a,b) = \overline \nu$ and denote by $(r,\theta)\in \R_+ \times ]-\pi,\pi]$ the polar coordinates of $(a,b)$, such that $a+ib = r e^{i\theta}$. Then $\overline \pi(a,b) = (a,b)$ if and only if
\begin{eqnarray*}
r e^{i\theta} & =& \frac{\int e^{iz} \exp\po{-  U(z)+\rho r  \cos(z-\theta)}\pf\dd z}{\int  \exp\po{-  U(z)+\rho r  \cos(z-\theta)}\pf\dd z}
\end{eqnarray*}
which, multiplying both sides by $e^{-i\theta}$ and taking their imaginary part, implies
\begin{eqnarray*}
0 & =& \int_{-\pi}^\pi \sin(z-\theta) \exp\po{-  U(z)+\rho r  \cos(z-\theta)}\pf\dd z \\
& =& \int_{-\pi}^\pi \sin(z) \exp\po{  \cos(2z)\cos(2\theta) -\sin(2z)\sin(2\theta) +\rho r  \cos(z)}\pf\dd z.
\end{eqnarray*}
The change of variable $z \rightarrow - z$ on $(-\pi,0)$ yields
\begin{eqnarray*}
 0 & =& 2 \int_{0}^\pi \sin(z) \sinh\po \sin(2z)\sin(2\theta)\pf \exp\po{  \cos(2z)\cos(2\theta)  +\rho r  \cos(z)}\pf\dd z
\end{eqnarray*}
and the change of variable $z \rightarrow \pi - z$ on $(\pi/2,\pi)$ yields
\begin{eqnarray*}
 0 & =& 4 \int_{0}^{\pi/2} \sin(z) \sinh\po \sin(2z)\sin(2\theta)\pf \sinh\po \rho r  \cos(z)\pf \exp\po{  \cos(2z)\cos(2\theta)}\pf\dd z.
\end{eqnarray*}
If $\sin(2\theta)\geqslant 0$ (resp. $\leqslant 0$), then the integrand is positive (resp.  negative), which means that in fact it vanishes for all $z\in (0,\pi/2)$, namely
\begin{eqnarray*}
\sinh\po \sin(2z)\sin(2\theta)\pf \sinh\po \rho r  \cos(z)\pf &=& 0\hspace{20pt}\forall z\in(0,\pi/2).
\end{eqnarray*}
Finally, either $r = 0$, or $\theta = 0 \ mod\ \pi/2$, so that in both cases $\nu \in A_0 \cup A_0'$.
\end{proof}

In the following, $U(z) = -\cos(2z)$. Let $\rho_1 = \po \int \cos^2 \dd m_U\pf^{-1}$ and $\rho_2 = \po \int \sin^2 \dd m_U\pf^{-1}$. Note that
\[\rho_1 \ < \ 2 \ < \ \rho_2.\]
So far, we have proved that $m_U$ is an unstable equilibrium of $\Psi$ as soon as $\rho>\rho_1$, and that
\begin{eqnarray*}
\rho \leqslant \rho_1 & \Rightarrow & Fix(\pi) = \{m_U\}\\
\rho_1 < \rho \leqslant \rho_2 & \Rightarrow & Fix(\pi) = \{m_U,\overline \pi(a_*,0),\overline \pi(-a_*,0) \}\\
\rho_2 < \rho & \Rightarrow & Fix(\pi) = \{m_U,\overline \pi(a_*,0),\overline \pi(-a_*,0),\overline \pi(0,b_*),\overline \pi(0,-b_*) \}
\end{eqnarray*}
where $\rho \mapsto (a_*,b_*)=\po a_*(\rho),b_*(\rho)\pf \in (0,\infty)^2$ are  given by Proposition \ref{PropDeuxpuits}.
\begin{lem}
For $\rho>\rho_2$, $\overline \pi(0,b_*)$ and $\overline \pi(0,-b_*)$ are saddles for the flow $\Psi$.
\end{lem}

\begin{proof}
The Jacobian of the vector field $\overline F(a,b)$ at point $(0,b_*)$ is 
\begin{eqnarray*}
J_{\overline F}(0,b_*) &=&  \begin{pmatrix}
\rho\int \cos^2 \dd \overline \pi(0,b_*)-1 & 0\\
0& \rho\int \sin^2 \overline \pi(0,b_*) -1
\end{pmatrix},
\end{eqnarray*}
and the lemma will be proved when we will have established that $\rho\int \cos^2 \dd \overline \pi(0,b_*)>1$. In the rest of the proof, we write $\overline\pi_\rho(b) = \overline\pi_\rho(0,b) $ and
\begin{eqnarray*}
A(\rho) & = & \int \cos^2 \dd \overline\pi_\rho(b_*(\rho)).
\end{eqnarray*}

First, we note that $\partial_\rho b_* >0$, which can be seen  as follows: writing, for $b>0$,
\begin{eqnarray*}
g(\rho,b) & =& \int \sin \dd \overline\pi_\rho(b),
\end{eqnarray*}
we compute
\begin{eqnarray*}
\partial_\rho g(\rho,b) & =& b \po \int \sin^2 \dd \overline\pi_\rho(b) - \po  \int \sin \dd \overline\pi_\rho(b)\pf^2\pf \ >\ 0.
\end{eqnarray*}
On the other hand, $b_*(\rho)$ being a sink of the flow $\partial_t b_t = g(\rho,b_t) - b_t$, if $b < b_*(\rho)$ then $b < g(\rho,b)$. Thus, if $\tilde \rho > \rho$,
\[b \ <\ g(\rho,b)\ <\ g(\tilde \rho,b), \]
which means $b\neq b_*(\tilde \rho)$ for all $b<b_*(\rho)$. In other words, $b_*(\tilde \rho) \geqslant b_*(\rho)$.

Second, differentiating the relation $b_* = g(\rho,b_*)$, we obtain
\begin{eqnarray*}
\partial_\rho b_* & = & \partial_\rho \po \rho b_*\pf\po \int \sin^2 \dd \overline\pi_\rho(b_*) -    b_* ^2\pf\\
\Rightarrow\hspace{20pt} \partial_\rho b_* & = & \frac{b_*\po \int \sin^2 \dd \overline\pi_\rho(b_*) -    b_* ^2\pf}{1-\rho\po  \int \sin^2 \dd \overline\pi_\rho(b_*) -    b_* ^2\pf}.
\end{eqnarray*}
In particular, $\partial_\rho b_*>0$ implies $\rho\po  \int \sin^2 \dd \overline\pi_\rho(b_*) -  \po b_*\pf^2\pf<1$. Then
\begin{eqnarray*}
\partial_\rho\po \rho b_*(\rho)\pf & = & b_* + \frac{\rho b_*\po 1 - A(\rho)-    b_* ^2\pf}{1-\rho\po 1-A(\rho) -   b_*^2\pf}\\
& = & \frac{ b_*}{1-\rho\po 1-A(\rho) -  b_*^2\pf}.
\end{eqnarray*}
We compute
\begin{eqnarray*}
\partial_\rho\po \rho A(\rho)\pf & =& A(\rho) + \rho \partial_\rho(\rho b_*) \po \int \cos^2 \sin \dd\overline \pi_z\po b_*\pf - b_* A(\rho)\pf .
\end{eqnarray*}
Integrating by part,
\begin{eqnarray*}
\int \cos^2(z) \sin(z) e^{ \cos(2z) + \rho b \sin(z) } \dd z& =& \int \cos^2(z) \sin(z) e^{ 1 - 2\po \sin(z) - \frac14\rho b\pf^2 + \frac18\po\rho b\pf^2  }\dd z\\
& =& - \frac14 \int \sin(z) e^{ 1 - 2\po \sin(z) - \frac14\rho b\pf^2 + \frac18\po\rho b\pf^2  }\dd z\\
& & + \frac{\rho b}{4}\int \cos^2(z) e^{ 1 - 2\po \sin(z) - \frac14\rho b\pf^2 + \frac18\po\rho b\pf^2  }\dd z.
\end{eqnarray*}
This yields
\begin{eqnarray*}
\partial_\rho\po \rho A(\rho)\pf & =& A(\rho) + \frac{\rho  b_*^2}{1+ \rho A(\rho) -\rho\po 1 -  b_*^2\pf}\po - \frac14   + \frac14\rho   A(\rho) -  A(\rho)\pf \\
& =&  A(\rho) + \frac{\rho b_*^2}4\po1+\frac{-2+\rho\po 1 -  b_*^2\pf-4 A(\rho)}{1+ \rho A(\rho) -\rho\po 1 -  b_*^2\pf}\pf.
\end{eqnarray*}
In particular, in the case where $\rho A(\rho) \in (0,1]$, $\rho\po 1 -   b_*^2\pf < 1 + \rho A(\rho) < 2$, so that 
\begin{eqnarray*}
\frac{-2+\rho\po 1 -   b_*^2\pf-4 A(\rho)}{1+ \rho A(\rho) -\rho\po 1 -  b_*^2\pf} & \leqslant & \frac{-2+\rho\po 1 -  b_*^2\pf-4 A(\rho)}{2 -\rho\po 1 -  b_*^2\pf}.
\end{eqnarray*}
It means that, if $\rho A(\rho) \leqslant 1$, then
\begin{eqnarray*}
\partial_\rho\po \rho A(\rho)\pf & \leqslant & A(\rho) - \frac{A(\rho) \rho b_*^2}{2 -\rho\po 1 -   b_*^2\pf}\\
& = & A(\rho)\frac{ 2-\rho}{2 -\rho\po 1 -   b_*^2\pf}\ <\ 0
\end{eqnarray*}
as $\rho >\rho_2 >2$. It means that, if there exists $\rho_3>\rho_2$ such that $\rho_3 A(\rho_3) \leqslant 1$, then $\rho \mapsto \rho A(\rho)$ is strictly decreasing for $\rho\geqslant \rho_3$. This is in contradiction with the fact that $\rho A(\rho)$ converges to 1 as $\rho$ goes to infinity. Indeed, by a Laplace method at point $z=\frac\pi2$, 
\begin{eqnarray*}
h(\rho,b) & := & \frac{\int \cos^2(z) e^{\cos(2z)+\rho b \sin(z)}\dd z}{\int  e^{\cos(2z)+\rho b \sin(z)}\dd z}\\
& \underset{\rho\rightarrow\infty}\simeq & \frac{\int_{\R} u^2 e^{2 u^2-\frac12\rho bu^2}\dd u}{\int_{\R}  e^{2 u^2-\frac12\rho bu^2}\dd u}\\
& = & \frac{1}{\rho b - 4}.
\end{eqnarray*}
Note that, $b_*$ increasing with $\rho$, $\overline \pi_\rho \po b_*\pf$ converges as $\rho$ goes to infinity to a Dirac measure at point $\frac\pi2$, so that $b_* \rightarrow 1$. As a consequence,
\[\rho A(\rho) \ = \ \rho h(\rho,b_*)\ \underset{\rho\rightarrow\infty}\simeq \ \frac{\rho}{\rho - 4} \ \underset{\rho\rightarrow\infty}\longrightarrow \ 1. \]
Hence, for all $\rho>\rho_2$, $\rho A(\rho) > 1$, and $(0,b_*)$ is an unstable equilibrium point of the flow induced by $\overline F$, which concludes.
\end{proof}

\begin{proof}[Proof of Theorem \ref{ThmWQuadra}, point 2.]
According to Theorem \ref{ThmFix}, since $Fix(\pi)$ contains a finite number of points, $\mu_t$ necessarily converges to one of those.  For $\rho\leqslant \rho_c := \rho_1$, $Fix(\pi)=\{m_U\}$, hence $\mu_t$ converges to $m_U$. For $\rho > \rho_c$, all the fixed points which are not $(\pm a_*,0)$ are saddles for the flow $\Psi$, and Theorem \ref{ThmPasCVsaddle} implies that the probability that $\mu_t$ converges to one of them is 0. Thus, $\mu_t$ converges to $(\kappa a_*,0)$ for some random $\kappa\in\{-1,1\}$. By  Theorem \ref{ThmPasCVsaddle}, it means either $(a_*,0)$ or $(-a_*,0)$ is a sink for the flow $\Psi$, and by symmetry they are both sinks, so that, by Theorem \ref{ThmCVsink}, both have a positive probability to be chosen.
 \end{proof}

\subsection{Large attraction in a multi-well potential}\label{Multiwell}

In this section, we suppose that $U$ admits a non-degenerate minimum at $x_0\in\mathbb T$. 
We use the notation
\begin{eqnarray*}
\overline \pi_{\rho}(r,\theta)  & = & \overline \pi_{\rho}(a,b) 
\end{eqnarray*}
when $a+ib = re^{i\theta}$ (there will be no ambiguity). The Laplace method shows that, for a function $f\in\mathcal C(\mathbb T)$,
\begin{eqnarray*}
\int_{-\pi}^\pi f(z) e^{\rho r (\cos(z-\theta)-1)} \dd z & =& f(\theta)\sqrt{\frac{2\pi}{\rho r}}  + \underset{\rho\rightarrow\infty}o\po\frac1{\sqrt\rho}\pf.
\end{eqnarray*}
\begin{lem}\label{LemLaplace}
Suppose that $f\in\mathcal C^3(\mathbb T)$ with $f(\theta)=0$, and let $r_0>0$. Then
\begin{eqnarray*}
\int_{-\pi}^\pi f(z) e^{\rho r (\cos(z-\theta)-1)} \dd z & =& f''(\theta)\sqrt{\frac{\pi}{2(\rho r) ^{3}}}  + \underset{\rho\rightarrow\infty}o\po \rho  ^{-\frac32}  \pf,
\end{eqnarray*}
where the error term only depends on $\| f^{(j)}\|_\infty$, $j=0,1,2,3$, and is uniform in $r\in[r_0,1]$.
\end{lem}
\begin{proof}
For $\rho\geqslant 1$, let $\delta = \delta(\rho ) =\rho^{-\frac13} $. Without loss of generality, we take $\theta = 0$. First,
\begin{eqnarray*}
\left|\int_{|z|>\delta} f(z) e^{\rho r (\cos(z)-1)} \dd z\right| &  \leqslant & \| f\|_\infty e^{-\frac{r\rho\delta}{\pi^2}},
\end{eqnarray*}
and
\begin{eqnarray*}
\int_{|z|<\delta} \left| f(z) - z f'(\theta) - \frac{z^2}{2} f''(\theta)\right| e^{\rho r (\cos(z)-1)} \dd z &  \leqslant & \frac{\delta}6\| f^{(3)}\|_\infty\int_{|z|<\delta} z^2 e^{\rho r (\cos(z)-1)} \dd z . 
\end{eqnarray*}
By symmetry, $\int_{|z|<\delta}  z  e^{\rho r (\cos(z)-1)} \dd z = 0$. Finally, 
\[\int_{|z|<\delta}  z^2    e^{-\rho r\po \frac{z^2}{2}  +\frac{\delta^4}6\pf  } \dd z \ \leqslant\ \int_{|z|<\delta}  z^2    e^{\rho r (\cos(z)-1)} \dd z  \ \leqslant\ \int_{|z|<\delta}  z^2    e^{-\rho r\po \frac{z^2}{2}  -\frac{\delta^4}6\pf  } \dd z\]
so that, using that $\rho r\delta^3\leqslant 1$,
\begin{eqnarray*}
\left| \int_{|z|<\delta}  z^2    e^{\rho r (\cos(z)-1)} \dd z - \frac{\sqrt{2\pi}}{(\rho r) ^{\frac32}}  \right| & \leqslant & \frac{e}6\rho r \delta^4 \int_{|z|<\delta}  z^2    e^{-\rho r  \frac{z^2}{2}     } \dd z +  \int_{|z|>\delta}  z^2    e^{-\rho r  \frac{z^2}{2}     } \dd z \\
& = &  \underset{\rho\rightarrow\infty}o\po \rho ^{-\frac32} \pf.
\end{eqnarray*}
\end{proof}

Recall that, for $\nu \in \mathcal P\po\mathbb T\pf$ with a Lebesgue density (still denoted $\nu$), the free energy is defined as
\begin{eqnarray*}
J(\nu) &=& \int   U(x)\nu(\dd x) - \frac{\rho}2 \int \cos(x-z)\nu(\dd x)\nu(\dd z) + \int   \ln\po \nu(x)\pf \nu(\dd x).
\end{eqnarray*}

\begin{lem}\label{LemMultiJ}
For all $r_0>0$, uniformly on $\theta\in\mathbb T$ and $r\in[r_0,1]$,
\begin{eqnarray*}
J\po\overline\pi_\rho(r,\theta)\pf - J\po\overline\pi_\rho(1,x_0)\pf & \underset{\rho\rightarrow\infty}\longrightarrow &  U(\theta) - U(x_0)    + \frac12\po \frac1{r} - 1 +   \ln(r)\pf. 
\end{eqnarray*}
\end{lem}
\begin{proof}
We compute
\begin{eqnarray*}
J\po\overline \pi_\rho(r,\theta)\pf & =& - \frac\rho2  \po \int  \cos(x-\theta) \overline \pi_{\rho}(r,\theta) \po\dd x\pf\pf^2 - \frac\rho2   \po \int  \sin(x-\theta) \overline \pi_{\rho}(r,\theta) \po\dd x\pf\pf^2 \\
& & + \rho r \int \po \cos(z-\theta)-1\pf \overline \pi_{\rho}(r,\theta)(\dd z) - \ln \int e^{-U(z)+\rho r \po \cos(z-\theta)-1\pf}\dd z. 
\end{eqnarray*}
According to Lemma \ref{LemLaplace} and the Laplace method,
\begin{eqnarray*}
 \rho r \int \po \cos(z-\theta)-1\pf \overline \pi_{\rho}(r,\theta)(\dd z)   & \underset{\rho\rightarrow\infty}\simeq &  - \rho r \frac{ e^{U(\theta)}}{2\rho r e^{U(\theta)}} \ = \ -\frac12,
 \end{eqnarray*}
and similarly,
 \begin{eqnarray*}
 - \frac\rho2  \po \int  \po\cos(x-\theta)-1\pf \overline \pi_{\rho}(r,\theta) \po\dd x\pf\pf^2 - \frac\rho2   \po \int  \sin(x-\theta) \overline \pi_{\rho}(r,\theta) \po\dd x\pf\pf^2   & = &  \rho \times \underset{\rho\rightarrow\infty}{\mathcal O}\po  \rho^{-2}\pf 
\end{eqnarray*}
so that
\begin{eqnarray*}
 - \frac\rho2  \po \int  \cos(x-\theta) \overline \pi_{\rho}(r,\theta) \po\dd x\pf\pf^2 - \frac\rho2   \po \int  \sin(x-\theta) \overline \pi_{\rho}(r,\theta) \po\dd x\pf\pf^2
& =&  - \frac\rho2 + \frac1{2r} + \underset{\rho\rightarrow\infty}{o}\po 1\pf 
\end{eqnarray*}
Finally,
\begin{eqnarray*}
 \ln \po \frac{\int e^{-U(z)+\rho r \po \cos(z-\theta)-1\pf}\dd z}{\int e^{-U(z)+\rho \po \cos(z-x_0)-1\pf}\dd z}\pf  & \underset{\rho\rightarrow\infty}\longrightarrow &  U(x_0) -  U(\theta) -   \frac12\ln(r).
\end{eqnarray*}
\end{proof}

\begin{proof}[Proof of Theorem \ref{ThmWQuadra}, point 3]
Let $\varphi_0>0$ be small enough so that $x_0$ is the only minimum of $U$ in $[x_0-\varphi_0,x_0+\varphi_0]$ (which should be understood as an interval of $\mathbb T$). For $\varphi\in(0,\varphi_0]$, we call
\begin{eqnarray*}
\mathcal D_\varphi &= &\po \frac12,1\right]\times(x_0-\varphi,x_0+\varphi) \\
\partial \mathcal D_\varphi &= &\po\left[\frac12,1\right]\times\{x_0-\varphi,x_0+\varphi\} \pf \cup \po \left\{\frac12\right\}\times  (x_0-\varphi,x_0+\varphi)\pf.
\end{eqnarray*} 
In other words, seen in the complex unitary disk by $r,\theta\mapsto r e^{i\theta}=w$, $\mathcal D_\varphi$ is a sector centred at $x_0$ and of angle $2\varphi$ of the band $r_0<|w|\leqslant 1$, and $\partial \mathcal D_\varphi$ is its boundary (inside the unitary disk).

We fix $\delta>0$. From Lemma \ref{LemMultiJ}, we chose  $\varphi$ small enough and $\rho_0$ large enough so that
\begin{eqnarray*}
\eta & :=  &    \underset{(r,\theta)\in \partial \mathcal D_\varphi}\inf J\po \overline \pi_\rho(r,\theta)\pf - J\po \overline \pi_\rho(1,x_0)\pf\ > \ 0
\end{eqnarray*}
and
\begin{eqnarray*}
\underset{(r,\theta)\in  \mathcal D_\varphi}\sup \int dist^2_{\Su}(z,x_0) \overline \pi_\rho(r,\theta) & < & \delta 
\end{eqnarray*}
for all $\rho>\rho_0$. Since $J$ is non-decreasing along the flow $\Psi$,
\[B \ :=\ \left\{\nu\in\mathcal P\po\mathbb T\pf, \overline \nu\in  D_\varphi, J\po\nu\pf < J\po \overline \pi_\rho(1,x_0)\pf  + \frac\eta2\right \}\]
is a non-empty, $\Psi$-invariant, open set. Let $L(B)$ be the set of limit points of sequences $\po \Psi_{t_k}(\nu)\pf_{k\in\mathbb N}$ with $\nu\in B$. Then $L(B)$ is an attractor in the sense of \cite[Section 5.1 p.22]{Benaim99}, and \cite[Theorem 7.3]{Benaim99} together with our controllability result (Proposition \ref{PropControle}) imply that
\begin{eqnarray*}
\mathbb P\po d_w\po \mu_t,L(B)\pf \underset{t\rightarrow\infty}\longrightarrow 0\pf & > & 0
\end{eqnarray*}
 (see also \cite[Proposition 4.13]{BenaimLedouxRaimond}). In particular,
 \begin{eqnarray*}
\mathbb P\po \exists t_0>0, \ \mu_t \in  \mathcal D_\varphi \forall t>t_0 \pf & > & 0,
\end{eqnarray*}
which concludes.
\end{proof}

\subsection*{Acknowledgements}
The author would like to thank, on the ond hand, Michel Bena\"im and Carl-Eric Gauthier for fruitful discussions about asymptotic pseudotrajectories, and on the other hand Jean-Baptiste Bardet, Florent Malrieu and Pierre-André Zitt for their help in the quadratic interaction case. We acknowledge financial support from the Swiss National Science Foundation Grant 200020-149871/1, and the french ANR PIECE.

\bibliographystyle{plain}
\bibliography{biblio}

\end{document}